\RequirePackage[l2tabu, orthodox]{nag}

\documentclass[12pt]{amsart}
\usepackage{fullpage,url,amssymb,enumitem,colonequals}
\setenumerate{listparindent=\parindent} 
\usepackage[all]{xy} 
\usepackage{mathrsfs} 
\usepackage{csquotes}
\usepackage{breakurl}
\usepackage{tabularx} 
\usepackage{caption}
\usepackage{subcaption}

\usepackage{booktabs} 
\usepackage{graphicx}

\usepackage[dvipsnames]{xcolor}

\definecolor{amethyst}{rgb}{0.6, 0.4, 0.8}
\usepackage[pagebackref=true, colorlinks, allcolors=amethyst]{hyperref}

\newcommand{\defi}[1]{\textsf{#1}} 

\newcommand{\A}{\mathbf{A}}
\newcommand{\C}{\mathbf{C}}
\newcommand{\F}{\mathbf{F}}
\newcommand{\G}{\mathbf{G}}

\newcommand{\N}{\mathbf{N}}
\renewcommand{\P}{\mathbf{P}}
\newcommand{\Q}{\mathbf{Q}}
\newcommand{\R}{\mathbf{R}}
\newcommand{\Z}{\mathbf{Z}}

\renewcommand{\aa}{\mathfrak{a}}

\newcommand{\pp}{\mathfrak{p}}
\newcommand{\qq}{\mathfrak{q}}
\newcommand{\nn}{\mathfrak{n}}

\newcommand{\calE}{\mathcal{E}}

\newcommand{\calH}{\mathcal{H}}
\newcommand{\calI}{\mathcal{I}}

\newcommand{\calL}{\mathcal{L}}

\newcommand{\calO}{\mathcal{O}}

\newcommand{\calX}{\mathcal{X}}

\DeclareMathOperator{\Div}{Div}
\DeclareMathOperator{\End}{End}

\DeclareMathOperator{\Gal}{Gal}

\DeclareMathOperator{\Hom}{Hom}

\DeclareMathOperator{\Nek}{Nek}

\DeclareMathOperator{\ord}{ord}

\DeclareMathOperator{\Spec}{Spec}

\DeclareMathOperator{\Vol}{Vol}

\newcommand{\ab}{{\operatorname{ab}}}
\newcommand{\ac}{{\operatorname{ac}}}

\newcommand{\alg}{{\operatorname{alg}}}

\newcommand{\cc}{{\operatorname{cc}}}
\newcommand{\cts}{{\operatorname{cts}}}

\newcommand{\cyc}{{\operatorname{cyc}}}

\newcommand{\GL}{\operatorname{GL}}

\newcommand{\NS}{\operatorname{NS}}

\newcommand{\MTT}{\operatorname{MTT}}
\newcommand{\new}{\operatorname{new}}

\newcommand{\New}{\operatorname{New}}
\newcommand{\T}{\mathbf{T}}
\newcommand{\old}{\operatorname{old}}
\renewcommand{\cc}{\mathfrak{c}}
\newcommand{\Symb}{\operatorname{Symb}}
\newcommand{\sgn}{\operatorname{sgn}}
\newcommand{\id}{\operatorname{id}}

\newtheorem{theorem}{Theorem}[section]
\newtheorem{lemma}[theorem]{Lemma}
\newtheorem{corollary}[theorem]{Corollary}
\newtheorem{proposition}[theorem]{Proposition}
\theoremstyle{definition}
\newtheorem{definition}[theorem]{Definition}

\newtheorem{example}[theorem]{Example}

\newtheorem{assumption}[theorem]{Assumption}

\theoremstyle{remark}
\newtheorem{remark}[theorem]{Remark}

\newtheorem{myalgorithm}[theorem]{Algorithm}

\numberwithin{equation}{subsection}

\usepackage{microtype}

\renewcommand*{\backrefalt}[4]{%
  \ifcase #1 %
    \relax
  \or
    $\uparrow$#2.%
  \else
    $\uparrow$#2.%
  \fi%
}

\begin{document}
\title[]{Quadratic Chabauty and $p$-adic Gross--Zagier}
\author{Sachi Hashimoto}
\thanks{While preparing this work the author was supported by National Science Foundation grant DGE-1840990.}
\address{Sachi Hashimoto, Max Planck Institute for Mathematics in the Sciences, Inselstrasse 22, 04103 Leipzig, Germany}
\email{sachi.hashimoto@mis.mpg.de}

\begin{abstract}
Let $X$ be a quotient of the modular curve $X_0(N)$ whose Jacobian $J_X$ is a simple factor of $J_0(N)^{\new}$ over $\Q$. Let $f$ be the newform of level $N$ and weight $2$ associated with $J_X$; assume $f$ has analytic rank $1$.
We give analytic methods for determining the rational points of $X$ using quadratic Chabauty by computing two $p$-adic Gross--Zagier formulas for $f$. Quadratic Chabauty requires a supply of rational points on the curve or its Jacobian; this new method eliminates this requirement. To achieve this, we give an algorithm to compute the special value of the anticyclotomic $p$-adic $L$-function of $f$ constructed by Bertolini, Darmon, and Prasanna, which lies outside of the range of interpolation. 
\end{abstract}

\maketitle

\section{Introduction}
Let $X$ be a smooth projective geometrically integral curve of genus $g >1$ over $\Q$. Let $p$ be a prime of good reduction for $X$.
Chabauty's method is a family of $p$-adic methods developed to try to determine $X(\Q)$; the goal is to show that $X(\Q)$ lies in the zero set of nontrivial locally analytic functions from $X(\Q_p)$ to $\Q_p$. Since a locally analytic function has only finitely many zeros on each residue disk of $X(\Q_p)$, this exhibits $X(\Q)$ inside a finite set.

Let $\rho(J_X)$ denote the N\'eron--Severi rank of the Jacobian $J_X$ of $X$ and $r$ the Mordell--Weil rank of $J_X(\Q)$ over $\Q$.
When $r < g + \rho(J_X) - 1$,  the quadratic Chabauty method \cite{QCIntegral,QCI,QCII} makes effective M.~Kim's program \cite{KimUnipotentAlbanese} for an explicit Faltings's theorem by studying quotients of the unipotent fundamental group at depth $2$. It uses $p$-adic height functions to determine a finite set containing $X(\Q)$. Crucially, the construction of the locally analytic function using $p$-adic heights requires knowing sufficiently many rational points on $X$ or $J_X$ \cite[Section~3.3]{examplesandalg}.

We develop a quadratic Chabauty method for certain quotients of modular curves that replaces this requirement for knowing rational points on $X$ or $J_X$ with computations of special values of $p$-adic $L$-functions. Our methods apply to quotients $X$ of $X_0(N)$ whose Jacobians are simple quotients of $J_0(N)^{\new}$ over $\Q$. 

In particular, we consider the case when $J_X$ is a simple factor of $J_0(N)^{\new}$. 
Then $J_X$ has the property that $ g = \rho(J_X)$ and $r$ is a multiple of $g$. By requiring $r =  g\geq 2$, the inequality $r < g + \rho(J_X) - 1$, is satisfied and we may apply quadratic Chabauty. Studying the isogeny decomposition of $J_0(N)^{\new}$ shows that the $L$-function of $J_X$ can be written in terms of the $L$-function of a newform $f_X$ of weight $2$ and level $N$ and its Galois conjugates. In this way, we can rephrase our study of $J_X$ as a study of the newform $f_X$ and its $L$-function.

Our main theorem is Theorem~\ref{thm:mainthm} where we give a new analytic method for computing the quadratic Chabauty function on these quotients in terms of two special values of $p$-adic $L$-functions. The quadratic Chabauty function can be written as the difference between a local and global $p$-adic height: this theorem gives a novel way of expressing the global height as a locally analytic function. We do not discuss the computation of the local height in this paper. To rewrite the global height, we use $p$-adic Gross--Zagier formulas, taking advantage of the fact that $X$ is a modular curve, to construct locally analytic functions for quadratic Chabauty. 

Our method proceeds by studying Rankin--Selberg $L$-functions associated with $f_X$ and an imaginary quadratic field $K$. This analytic study of the $L$-functions of modular forms contains deep arithmetic information through the computation of special values of $L$-functions. When $K$ satisfies the Heegner hypothesis, we can construct a Heegner divisor $y_K \in J_0(N)(K)$. The Gross--Zagier theorem \cite{GrossZagier1} relates the derivative of the Rankin--Selberg $L$-function $L'(f_X, \chi, 1)$ of a weight $2$ newform $f_X$ to the canonical height of the $f_X$-isotypical component of $y_K$ twisted by ring class characters $\chi$. 

Assume the prime $p$ is ordinary for the newform $f_X$. Perrin-Riou \cite{PerrinRiouInvent} developed a $p$-adic version of Gross and Zagier's formula, relating the derivative of a $p$-adic $L$-function to the cyclotomic $p$-adic height of the $f_X$-isotypical component of the Heegner divisor. This height was also studied in papers of Mazur--Tate \cite{MazurTateBiext}, Schneider \cite{Schneider}, and Coleman--Gross \cite{ColemanGross}. We give an algorithm to compute the $p$-adic height of the Heegner divisor using overconvergent modular symbols \cite{PSOverconvergent}.

Bertolini, Darmon, and Prasanna \cite{BDP1} constructed an anticyclotomic $p$-adic Rankin $L$-series $L_p$ that also interpolates the central values of the Rankin--Selberg $L$-function $L(f_X, \chi, 1)$ for a different set of Hecke characters of $K$.
The special value of the anticyclotomic $p$-adic $L$-function is proportional to $(\log_{f_X dq/q} y_K)^2$. We give an algorithm to compute this special value, which lies outside of the range of interpolation.

Background material and notational conventions are found in Section \ref{sec:background}. We compute the special value of Perrin-Riou's $p$-adic $L$-function Section \ref{Ch:Height}. Section \ref{Ch:Log} discusses the special value of the anticyclotomic $p$-adic $L$-function. Finally, Section \ref{ch:quadchab} gives the quadratic Chabauty algorithms and several examples of the method.

\section{Background and notation}
\label{sec:background} 
\subsection{Modular forms and modular curves}
\label{sec:algebraicmodularform}
For background on modular curves, we loosely follow \cite{BakerGJGPoonen} but take the perspective of Katz \cite[Section~1.1]{BDP1} on viewing modular forms as functions on marked elliptic curves with level structure.
Let $\calH$ denote the complex upper half plane. Throughout we use the convention that $z \colonequals x + i y$.

\begin{definition}
Let $R$ be a ring.
An \defi{elliptic curve $E$ with $\Gamma_0(N)$-level structure over $R$} is a pair $(E, C_N)$ where $E \to \Spec R$ is an elliptic curve over $\Spec R$ and $C_N$ a cyclic a finite locally free sub-group scheme of $E$ of order $N$, meaning that locally f.p.p.f on the base we can find a point $P$ such that $C_N = \sum_{a \mod N} [a P]$ as Cartier divisors.

A \defi{marked} elliptic curve $E$ with $\Gamma_0(N)$-level structure over $R$ is a triple $(E, C_N, \omega)$ such that $\omega$ is a nonzero global section of $\Omega^1_E$ over $\Spec R$ and $(E,C_N)$ is an elliptic curve with $\Gamma_0(N)$-level structure.
\end{definition}

\begin{definition}
Let $F$ be a field and $R$ an $F$-algebra.
A \defi{weakly holomorphic algebraic modular form $f$ of weight $k$ for $\Gamma_0(N)$ defined over $F$} is a rule that assigns to every isomorphism class of marked elliptic curves with $\Gamma_0(N)$-level structure $(E, C_N, \omega)$ over $R$ an element $f(E, C_N, \omega) \in R$ such that
\begin{enumerate}
\item for every homomorphism of $F$-algebras $j: R\to R'$, $f((E, C_N, \omega)\otimes_j R') = j(f(E, C_N, \omega))$;
\item for all $\lambda \in R^\times$, $f(E, C_N, \lambda \omega) = \lambda^{-k} (E, C_N, \omega)$.
\end{enumerate}
\end{definition}

Denote by $(\mathrm{Tate}(q), T_N, du/u)$ the Tate curve  $\G_m/q^\Z$  with $\Gamma_0(N)$-structure $T_N$ and canonical differential $du/u$ where $u$ is the usual parameter in $\G_m$. The Tate curve is defined over $F((q^{1/d}))$ for some $d|N$.

\begin{definition}
An \defi{algebraic modular form $f$ of weight $k$ for $\Gamma_0(N)$ defined over $F$} is a weakly holomorphic one such that $f(\mathrm{Tate}(q), T_N, du/u)  \in F[[q^{1/d}]]$ for all $T_N$.
\end{definition}

Algebraic modular forms are sections of line bundles in the following way.  If $N \geq 3$, the modular curve $Y_0(N)$ is a moduli space: for any  $\Z[1/N]$-algebra $R$, the points $Y_0(N)(R)$ can be identified with the set of isomorphism classes of elliptic curves with $\Gamma_0(N)$-level structure over $R$. Let $\calE$ be the universal elliptic curve with $\Gamma_0(N)$-level structure and $\pi: \calE \to Y_0(N)$. Let $\underline{\omega} \colonequals \pi_* \Omega^1_{\calE/Y_0(N)}$ be the
 sheaf of relative differentials. If $g$ is a weakly holomorphic modular form of weight $k$ for $\Gamma_0(N)$, we view $g$ as a global section of $\underline{\omega}^k$ by $g(E, C_N) = g(E, C_N, \omega ) \omega^k$ where $\omega$ is a generator of $\Omega^1_{E/R}$.
 Furthermore, $\underline{\omega}$ extends to a line bundle on $X_0(N)$, characterized by the property that the global sections $H^0(X_0(N), \underline{\omega}^k)$ are exactly the space of weight $k$ modular forms for $\Gamma_0(N)$.

Specializing to the complex numbers, $X_0(N)(\C)$ is a compact Riemann surface and the map $\calH \to \Gamma_0(N) \backslash \calH$ sending 
\begin{align}
\label{eqn:tautoellcurve}
 \tau \mapsto ( \C / (\Z + \tau \Z), 1/N)
\end{align}
  identifies $Y_0(N)(\C)$ with $\Gamma_0(N) \backslash \calH$, where we identify $1/N$ with the cyclic subgroup it generates by abuse of notation.
\begin{definition}
\label{def:algmodformC}
 If $g$ is a weakly holomorphic modular form of weight $k$ it gives a holomorphic section of the sheaf $\underline{\omega}^k$ (viewed as an analytic sheaf on the Riemann surface $X_0(N)(\C)$) and gives rise to a holomorphic function on $\calH$ by the definition
 \begin{align}
 \label{def:gtau}
  g(\tau) \colonequals g(\C/(\Z+ \tau \Z) , 1/N, 2 \pi i z)
  \end{align}
 where $z$ is the usual coordinate on $\C/(\Z+ \tau \Z)$.
 \end{definition}

Let $S_2(N)$ denote the weight $2$ cuspforms for $\Gamma_0(N)$. We can decompose
\[S_2(N) =  S_2(N)^{\old }\oplus S_2(N)^{\new} .\] The space $S_2(N)$ is equipped with an action of Hecke operators $T_n$ for $n \geq 2$, and we write $\T = \Z[T_2, T_3, \dots]$ for the Hecke algebra.
Furthermore, $S_2(N)$ has a basis of cupsforms $f$ that are eigenforms for all the Hecke operators, that is $T_n f = a_n(f) f$  for all $n \geq 2$. We say $f \in  S_2(N)^{\new}$ is a \defi{newform} if it is an eigenform and $a_1(f) = 1$. Write $\New_N$ for the set of newforms of level $N$. If $f \in \New_N$, then $f$ is defined over a number field which we denote by $E_f \colonequals \Q(a_2, a_3, \dots)$. There is action of $G_{\Q} \colonequals \Gal(\overline{\Q}/\Q)$ on the set of newforms.

Shimura attached to each newform $f$ a $\Q$-simple isogeny factor $A_f$ of $J_0(N)$ of dimension $[E_f:\Q]$. We have an isogeny decomposition over $\Q$
\[ J_0(N) \sim \bigoplus_{M|N}  \bigoplus_{f \in G_{\Q} \backslash \New_M} A_f^{\sigma_0(N/M)}\] where $\sigma_0(n)$ is the divisor function.

\begin{definition}
Let $X/\Q$ be a smooth projective geometrically integral curve.
We say $X$ is \defi{$\Gamma_0(N)$-modular} if there is a non-constant morphism $\phi: X_0(N) \to X$ over $\Q$.
\end{definition}


\begin{lemma}
\label{thm:GL2typepartII}
Suppose $X$ is $\Gamma_0(N)$-modular of genus $g$, and its Jacobian $J_X$ is simple over $\Q$. Then $g = \dim J_X = \dim \End_{\Q}^0(J_X)$ and there exists an integer $M|N$ and $f_X \in \New_M$ such that $J_X$ is $\Q$-isogenous to $A_{f_X}$. Furthermore, $f_X$ is unique up to Galois conjugacy.
\end{lemma}

\begin{proof}
If $X$ is $\Gamma_0(N)$-modular, then $J_X$ is a quotient of $J_0(N)$ of dimension $g$. Then $J_X$ is simple. So
by the isogeny decomposition of $J_0(N)$ and Poincar\'e reducibility, $J_X$ is isogenous over $\Q$ to some $A_{f_X}$, with $f_X \in \New_M$ for some $M|N$, where $f_X$ is unique up to Galois conjugacy. Then $\dim J_X = \dim A_f = [E_f :\Q] = \dim \End_{\Q}^0(A_f) =\dim \End_{\Q}^0(J_X). $
\end{proof}

\begin{definition}
Let $\phi: X_0(N) \to X$ be $\Gamma_0(N)$-modular and genus $g$, and suppose $J_X$ is simple. By Lemma \ref{thm:GL2typepartII}, we have an associated $f_X\in \New_M$, unique up to Galois conjugacy.
We say $X$ is a \defi{simple new $\Gamma_0(N)$-modular} curve if $X$ satisfies all these assumptions and furthermore $f_X \in \New_N$ (that is, $M = N$).
\end{definition}

\begin{assumption}
We assume throughout the whole paper that $\phi: X_0(N)\to X$ is a simple new $\Gamma_0(N)$-modular curve with Jacobian $J_X$ associated to a Galois orbit of newforms $\{ f_X^{\sigma} : \sigma \in \Gal( E_{f_X}/\Q) \}$. We write $\pi: J_0(N) \to J_X$ for the map induced by $\phi_*$.
\end{assumption}

\subsection{Heegner points}
\label{sec:Heegnerpoints}
We now introduce definitions and notation for Heegner points and discuss relevant background.
More details on Heegner points and references for this section can be found in \cite{GrossX0N,GrossZagier1,GrossKohnenZagier}.

Let $K$ be an imaginary quadratic field of class number one. Let $N$ be a positive integer.
\begin{definition}
The \defi{Heegner hypothesis} for $K$ and $N$ is the assumption that every prime $q |N$ splits in $K$.
\end{definition}

If $N$ satisfies the Heegner hypothesis then we can write $(N) = \nn \overline{\nn}$ in $\calO_K$. Write $\nn= \Z N + \Z \frac{b+ \sqrt{D}}{2}$ for some $b \in \Z$. 
Then under the map \eqref{eqn:tautoellcurve}, the point
\begin{align}
\label{eqn:CMpt}
\tau_\nn \colonequals  \frac{ b+ \sqrt{D}}{2N}
\end{align}
corresponds to the elliptic curve with $\Gamma_0(N)$-level structure $(\C/\overline{\nn}^{-1}, 1/N)$. Then $\C/\overline{\nn}^{-1}$ is an elliptic curve with CM by $\calO_K$ and hence has a model $A$ defined over $\calO_K$.
\begin{definition}
We define the \defi{Heegner point} to be the point $P_K \colonequals (A, A[\nn]) \in X_0(N)(K)$.
The point $P_K$ gives rise to a divisor class $y_K \colonequals [P_K - \infty] \in J_0(N)(K)$. 
\end{definition}
We now discuss the choice of a differential form for $A$ in order to evaluate modular forms at the Heegner point.
Choose $\omega_A$ a N\'eron differential on $A$.
Then the period lattice of $\omega_A$ is $\Omega_K \cdot \calO_K \subset \C$.
Then for any modular form $g$ of weight $k$ for $\Gamma_0(N)$ we have
\begin{align}
g(A, A[\nn], \omega_A) = g(\C/(\Z + \Z \tau_\nn), 1/N, \Omega_K \overline{\alpha} dz) =   \frac{g(\tau_\nn)}{ (\overline{\alpha} \Omega_K )^k} 
\end{align} 
where $\overline{\alpha}$ is a generator of $\overline{\nn}$.
Let $\Omega_A$ be 
$1/(2 \pi i)$ times 
the real period of $A$. 
By  \cite[p.25]{BDPconiveau} these are related by $\Omega_K = \Omega_A/\sqrt{D}$.
\begin{remark}
Note that $\Omega_A$ does not depend on the choice of $N$, only on $K$.
\end{remark}

Consider the vector space $V = J_0(N)(K) \otimes \overline{\Q}$. The height pairing gives an inner product on $V$ for which the Hecke operators are self-adjoint: the adjoint of an endomorphism of an abelian variety with respect to the height pairing is the Rosati involution by \cite[(3.4.3)]{MazurTateBiext} c.f. Proposition~\ref{prop:ProjectionFormula}. Since the Hecke algebra on $X_0(N)$ is a product of totally real fields, the Rosati involution (which is totally positive) is forced to be the identity.

By the spectral theorem, the fact that the Hecke operators are self-adjoint for the height pairing yields a decomposition into eigenspaces
\begin{align}
V = \bigoplus_{f \in \New_N} V^f.
\end{align}

\begin{definition}
Let $f \in \New_N$ and $\sigma \in \Gal(E_f/\Q)$. We denote by $y_{K,f^\sigma }$ the $f^\sigma$-isotypical component of $y_K$, with $y_{K,f^\sigma} \in  V^{f^\sigma} \subset J_0(N)(K)\otimes E_f$.
\end{definition}
Because distinct eigenvectors are orthogonal, for any $\sigma, \tau \in \Gal(E_f/\Q)$
\begin{align*}
\langle y_{K,f^\sigma} , y_{K,f^\tau} \rangle = 0
\end{align*}
whenever $\sigma \neq \tau$.

Let $f \in \New_N$. The Gross--Zagier formula \cite[Theorem~I.6.3]{GrossZagier1} says
\[L'(f, 1) L(f^\varepsilon, 1) = h_{\text{NT}}(y_{K,f})\]
where $ h_{\text{NT}}$ denotes the real valued canonical height on $J_0(N)$.

A theorem of Waldspurger then guarantees infinitely many $K$ such that $L(f^\varepsilon, 1) \neq 0$ by relating the value $L(f^\varepsilon, 1)$ to the $|D|$th Fourier coefficient of another modular form of weight $3/2$ obtained via the Shimura correspondence.
We do not know an explicit way to construct a suitable $K$ a priori; instead we can verify this condition computationally.

The Fricke involution $w_N$ on $X_0(N)$ induces an involution on $S_2(N)$. Since $f$ has analytic rank $1$, we have that $L'(f, 1) \neq 0 $ only if the Fricke sign is $1$,  so we require this.

The morphism $\phi: X_0(N) \to X$ induces a morphism $\pi: J_0(N) \to J_X$ via pushforward.
The action of complex conjugation on $\pi(y_K)$ is given by 
\[\pi(w_N(P_K) - \infty)\] 
therefore, $\pi(y_{K})\in J_X(\Q)$ \cite[(5.2)]{GrossX0N}.
The Hecke action on $J_X$ can be identified with an order $\calO_{f_X}$ in $K$. Finally, $\calO_{f_X} \pi(y_{K})$ generates a finite index subgroup of $J_X(\Q)$ \cite[7~Appendix]{DograLeFourn}.

\begin{remark}
We will often denote (Galois orbits of) newforms by labels. These labels are LMFDB labels \cite{lmfdb}. 
Where appropriate, we will also reference curves with LMFDB labels. All of the labels in this paper are LMFDB labels.
\end{remark}

\subsection{Hecke characters}
\label{sec:heckechars}
We give a short background section on Hecke characters and their varied incarnations in this section. For this section we fix embeddings $\iota_p : \overline{\Q} \hookrightarrow \C_p$ and $\iota_\infty : \overline{\Q} \hookrightarrow \C$. Recall that $K$ is still an imaginary quadratic field.

Let $\cc$ be an integral ideal of $K$. Let $I_\cc$ denote the group of fractional ideals of $K$ prime to $\cc$. Let $J_\cc$ be the set of ideals in $K$ satisfying 
\[J_\cc = \{ (\alpha) : \text{ for all  prime ideals } \qq | \cc, v_{\qq}(\alpha-1) \geq \ord_{\qq} (\cc)\} \subset I_\cc.\]
\begin{definition}
Let $(n_1, n_2) \in \Z^2$.
An \defi{algebraic Hecke character of infinity type $(n_1, n_2)$ and conductor dividing $\cc$} is a homomorphism 
\[ \chi: I_\cc \to \C^\times \]
 such that \[\chi((\alpha)) = \alpha^{n_1} \overline{\alpha}^{n_2}, \,\,\, \text{ for all } \alpha \in J_\cc.\]
\end{definition}
 It is possible that $\chi$ can be extended to some Hecke character of conductor dividing $\cc'$; the smallest such integral ideal $\cc'$ is the \defi{conductor} of $\chi$. 

\begin{example}
Let $\qq$ be an $\calO_K$-ideal.
The norm character $\N_K$ sending $\qq \mapsto \#(\calO_K/\qq)$ has infinity type $(1,1)$ and conductor $\calO_K$.
\end{example}

We can associate to a Hecke character an id\`ele class character  $\chi: \A_K^\times / K^\times \to \C^\times$ 
such that $\chi_\infty(z) = z^{-n_1} \bar{z}^{-n_2}$, where $\chi_\infty$ denotes the component of $\chi$ at $(K \otimes \R)^\times $. The map $(K \otimes \R)^\times \to  \C^{\times}$ is constructed using the embedding $\iota_\infty$.

\begin{remark}
The sign convention for the infinity type is picked to agree with \cite{BDP1}, and is negative the convention in some other papers.
\end{remark}

Algebraic Hecke characters are in bijection with algebraic $p$-adic Hecke characters, as we now describe.
Let $p = \pp \overline{\pp}$ be a prime that splits in $K$.
A \defi{$p$-adic Hecke character} is a continuous homomorphism $\chi_{\Q_p}: \A_K^\times / K^\times \to \overline{\Q}_p^\times$. It is \defi{algebraic} if there are integers $n_1$ and $n_2$ such that the local factors $\chi_\pp$ on $K_\pp^\times \simeq \Q_p^\times$ and $\chi_{\bar{\pp}}$ on $K_{\bar{\pp}}^\times$ on $\Q_p^\times$ are of the form $\chi_\pp(z) = z^{-n_1}$ and $\chi_{\bar\pp}(z) = z^{-n_2}$.
Then $\chi_\C$ an algebraic Hecke character of infinity type $(n_1, n_2)$ corresponds to the $p$-adic Hecke character $\chi_{\Q_p}$ via the formula
\[ \chi_{\Q_p}(z) = \iota_p  \circ \iota_\infty^{-1}( \chi_\C(z)z_\infty^{n_1} \bar{z}_\infty^{n_2}) z_\pp^{-n_1} z_{\bar{\pp}}^{-n_2}\]
for an id\`ele $z = (z_v)$.

If $\chi_{\Q_p}$ is algebraic, then $\chi_{\Q_p}$ factors through $\Gal(K^{\ab}/K)$.  We call this \defi{the associated $p$-adic Galois representation} \[\chi_G:\Gal(K^{\ab}/K) \to \overline{\Q}_p^\times.\] 
 Conversely, given any $p$-adic Galois representation $\Gal(K^{\ab}/K) \to \overline{\Q}_p^\times$ we can obtain an id\`ele class character $\A_K^\times/K^\times  \to \overline{\Q}_p^\times$ by precomposing with the Artin map.

\section{The special value of the anticyclotomic \texorpdfstring{$p$}{p}-adic \texorpdfstring{$L$}{L}-function}
\label{Ch:Log}

In this section, we explain how to compute the special value of the $p$-adic Rankin $L$-series $L_p(f,1)$ introduced by Bertolini, Darmon, and Prasanna \cite{BDP1} attached to a newform $f \in S_2(N)$ and the imaginary quadratic field $K$ satisfying the hypotheses below. This value occurs at the norm character $\N$ with infinity type $(1,1)$. Since the norm character $\N$ lies outside of the range of interpolation for $L_p$, this value is not readily accessible by computing a classical $L$-value. We follow a method of Rubin \cite{RubinInvent} for evaluating the Katz $2$-variable $p$-adic $L$-function outside the range of interpolation. Our method requires us to first evaluate the $p$-adic $L$-function at certain characters in the range of interpolation. In the case of the $p$-adic Rankin $L$-series of Bertolini, Darmon, and Prasanna, if $\chi$ is a character in the range of interpolation, then the value of $L_p(f)$ at $\chi$ is shown to be an explicit multiple of the central value of the Rankin $L$-series $L(f, \chi, 1)$ as in \eqref{eqn:BDPinterpolation}. An explicit Waldspurger's formula (Theorem~\ref{thm:waldspurger}) relates the central $L$-values $L(f, \chi, 1)$ to the square of the Shimura--Maass derivative of $f$ at the Heegner point. By considering $p$-adic characters in the ``anticyclotomic'' direction, Bertolini, Darmon, and Prasanna obtain the special value formula, relating the square of the logarithm of $y_K$ to $L_p(f, 1)$. We will refer to $L_p(f)$ as the \defi{anticyclotomic $p$-adic $L$-function}.

The section is divided into two subsections. Section \ref{sec:inrangeinterp} is devoted to explaining how to compute the values of $L_p(f)$ in the range of interpolation. The strategy here is to compute the Shimura--Maass derivatives of $f$ evaluated at the Heegner point. Section \ref{sec:outrangeinterp} develops the computation of the special value $L_p(f, 1)$ following Rubin's method. The key proposition is Proposition \ref{MainProp}, which relates the value $L_p(f,1)$ to values inside the range of interpolation.

Let $f$ be a newform in $S_2(N)$ with coefficient field $E_f$. 
We start by collecting some running assumptions that will be used for the remainder of the section.

\begin{assumption}\label{hyp:logs}
\hfill
\begin{enumerate}
\item Let $p> 2$ be a prime number not dividing $N$; 
\item \label{hyp:Efsplitting} assume $p$ splits in $E_f$  and fix an embedding $e: E_f \to \Q_p$; 
\item assume $f$ has analytic rank $1$; \label{hyp:AnalyticRank}
\item let $K = \Q(\sqrt{D})$ be an imaginary quadratic field of class number $1$; \label{hyp:classnumber}
\item  assume $(p) = \pp \overline{\pp}$ splits in $K$  and $K$ has odd discriminant $D < - 3$; \label{hyp:ksplit}
\item assume \label{hyp:HeegnerHypo} every prime $q$ dividing $N$ splits in $K$ (\defi{the Heegner hypothesis}).
\end{enumerate}
\end{assumption}
 Assumptions \eqref{hyp:classnumber} and \eqref{hyp:Efsplitting} are simplifying assumptions to avoid working in field extensions of $\Q_p$ for ease of computation and clarity of exposition, \eqref{hyp:ksplit} is required for the construction of the $p$-adic $L$-function, and \eqref{hyp:HeegnerHypo}  ensures the existence of $y_K \colonequals [P_K - \infty] \in J_0(N)(K)$.  

Now, we are ready to state the main theorem of \cite[Main~Theorem]{BDP1}: they show that
\begin{align} 
\label{specialvalueBDP}
L_p(f,1) = \left( \frac{1 - a_p(f) +p}{p}\right)^{2}\log_{f dq/q}(y_K)^2,
\end{align}
where $\log_{f dq/q}$ denotes the logarithm on $J_0(N)$ and $L_p(f,1)$ denotes the value of $L_p(f)$ at the character $\N$ (the notation is to emphasize the similarity to the special value of Perrin-Riou's $p$-adic $L$-function, see Section \ref{Ch:Height}).
Our goal in this section is to provide a method for computing $L_p(f, 1)$.

\begin{remark}
\label{rem:logonAf}
Given a surjective morphism $\pi: J_0(N) \to J$, the logarithm on $J_0(N)$ induces an inclusion $\pi^*: H^0(J, \Omega^1) \hookrightarrow H^0(J_0(N), \Omega^1)$. Since $f dq/q$ is in the image of $\pi^*$, by applying change of variables \cite[Proposition~2.4 (iii)]{Torsion} the logarithm on $J_0(N)$ with respect to $f dq/q$ can be considered as the logarithm on $J$.
\end{remark}

We start by recalling the interpolation property of the  $p$-adic $L$-function $L_p$ in \cite{BDP1} associated to $f$. The interpolation property will provide enough information to work with the $p$-adic $L$-function for our purposes; for more details on the construction of the anticyclotomic $p$-adic $L$-function see the main reference \cite{BDP1} or \cite{TwoTrilogies} for an expository article. 
Let $K_\infty^{\ac} / K$ be the anticyclotomic $\Z_p$-extension of $K$, and $\Gamma^- \colonequals \Gal(K_\infty^{\ac} / K)$ denote its Galois group. Write $\hat{\calO}$ for the completion of the ring of integers of the maximal unramified extension of $\Q_p$ and define $\Lambda^\ac \colonequals \Z_p[[\Gamma^-]] \hat{\otimes}_{\Z_p} \hat{\calO}$.

Let  $S \subset \Hom_{\cts} (\Gamma^-, \overline{\Q}_p^\times)$ be the subset of Galois characters associated to Hecke characters of $K$ with infinity type $(1 + r, 1 -r)$ for some integer $r \geq 1$. Bertolini, Darmon, and Prasanna prove that there exists $L_p(f) \in \Lambda^\ac$ interpolating the algebraic $L$-values $L_{\alg}(f, \chi^{-1}, 0)$ for all $\chi \in S$ in the following sense:  
each $\chi \in S$ determines a map $\Lambda^\ac \to \widehat{\overline{\Q}}_p$ by $\hat{\calO}$-linear extension. The interpolation property \cite[(5.2.3)]{BDP1} for $L_p(f)$ says that for all $\chi \in S$ and $\Omega_p$ a $p$-adic period associated to $y_K$ defined in \cite[(5.2.2)]{BDP1}, we have
\begin{align}
\label{eqn:BDPinterpolation}
\chi(L_p(f))/\Omega_p^{4r} =  (1 - \chi^{-1}(\overline{\pp})a_p + \chi^{-2} (\overline{\pp}) p)^2 L_{\alg}(f, \chi^{-1}, 0).
\end{align}
The left hand side of \eqref{eqn:BDPinterpolation} is our notation for evaluating the $p$-adic $L$-function $L_p(f)$ at $\chi$ for $\chi$ in the range of interpolation.
The norm character $\N$, with infinity type $(1,1)$, is not in $S$, so we cannot use \eqref{eqn:BDPinterpolation} to evaluate ``$\N(L_p(f))$''.
However, for $\chi \in S$, we can evaluate $\chi(L_p(f))$, which we will now explain.

\begin{remark}
The value $L_{\alg}(f, \chi^{-1}, 0)$ is defined as an explicit constant multiple of the special value of the Rankin $L$-series $L(f, \chi^{-1}, 0)$. This $L$-series $L(f, \chi^{-1}, s)$ can be written explicitly (in some right half plane of $\C$) as an Euler product over prime ideals $\qq$ of $\calO_K$, 
\begin{align*}
L(f, \chi^{-1}, s) = \prod_{\qq} [(1- \alpha_{\N \qq}(f) \chi^{-1}(\qq)\N\qq^{-s})(1- \beta_{\N \qq}(f) \chi^{-1}(\qq)\N\qq^{-s})]^{-1}
\end{align*}
where $\alpha_q(f)$ and $\beta_q(f)$ are the roots of the Hecke polynomial $x^2 - a_q(f)x + q$ and if $\N \qq = q^t$ then we set $\alpha_{\N \qq} \colonequals \alpha_{q}(f)^t$ and $\beta_{\N \qq} \colonequals \beta_{q}(f)^t$. Using Rankin's method, one can show that $L(f, \chi^{-1}, s)$ has an analytic continuation to the entire complex plane.
 \end{remark}

\subsection{Evaluating inside the range of interpolation}
\label{sec:inrangeinterp}

In order to $p$-adically interpolate the values $L(f, \chi^{-1}, 0)$ and evaluate the anticyclotomic $p$-adic $L$-function, we interpret $L_{\alg}(f, \chi^{-1}, 0)$ as values of a $p$-adic modular form. This is done using an explicit form of  Waldspurger's formula \cite[Theorem~5.4]{BDP1}. We precede the statement of this formula with a discussion of the Shimura--Maass operator.

Recall that we view modular forms as global sections $H^0(X_0(N), \underline{\omega}^k)$ as in Section \ref{sec:algebraicmodularform} and Definition \ref{def:algmodformC}. We also rely on the background and notation on Heegner points in Section \ref{sec:Heegnerpoints}. In particular, recall that the marked elliptic curve with $\Gamma_0(N)$-torsion corresponding to the Heegner point in $P_K \in X_0(N)(K)$ is denoted by
\[(A, A[\nn], \omega_A) = (\C/(\Z + \Z \tau_\nn), 1/N, \Omega_K \bar{\alpha}  dz).\]

Let $M_k(\Gamma_0(N))$ denote the space of modular forms of weight $k$ and $g \in M_k(\Gamma_0(N))$.
\begin{definition}
The \defi{Shimura--Maass} derivative $\delta_k$ is defined as
\begin{align*}
\delta_k(g) = \frac{1}{2\pi i} \left( \frac{\partial }{\partial z} + \frac{k}{2 i y}   \right) g(z).
\end{align*}
\end{definition}

\begin{definition}
We say $g$ is a \defi{nearly holomorphic modular form of weight $k$ and order less than or equal to $P$} if $g$ is $C^\infty$ on $\calH$ and we can express $g$ as a sum
\begin{align*}
g(z) = \sum_{j=0}^P g_j(z)y^{-j}
\end{align*}
where $g_j(z)$ are holomorphic functions on $\calH$, the function $g$ transforms like a modular form of weight $k$ for $\Gamma_0(N)$, and $g$ has finite limit at the cusps. We denote the space of such forms $N^P_k(\Gamma_0(N))$.
\end{definition}

The following lemmas can be proved with simple calculations.
\begin{lemma}[{\cite[Lemma~2.1.3]{UrbanNearlyOC}}]
\label{lem:shimuramaass}
Let $g \in N^P_k(\Gamma_0(N))$
Assume $P> 2k$. There exist $g_0, \dots, g_P$ with $g_i \in M_{k-2i}(\Gamma_0(N))$ such that
\[ g = g_0 + \delta_{k-2} g_1 + \dots + \delta_{k-2P}^P g_P.\]
\end{lemma}

\begin{lemma}[{\cite[(52)]{Zagier123}}]
\label{prop:shimuramodular}
If $g \in M_k(\Gamma_0(N))$ and $\gamma \in \Gamma_0(N)$ such that $\gamma = \begin{pmatrix} a & b\\ c & d \end{pmatrix}$
then $\delta_k(\gamma g)= (c +d)^{k+2} \delta_k(g)$.
\end{lemma}

Thus, Lemma \ref{lem:shimuramaass} and Proposition  \ref{prop:shimuramodular} imply that the Shimura--Maass derivative is an operator $\delta_k: N^P_k(\Gamma_0(N)) \to N^P_{k+2}(\Gamma_0(N))$.
We denote by $\delta^r$ the $r$-fold composition $\delta_{k+2r-2}\circ \dots \circ \delta_k$. Let $d$ be the Atkin--Serre derivative that acts on $q$-expansions by $q d/ dq$. 
Shimura \cite{ShimuraAlgebraicity}  showed that the values of $\delta^r f$ and $d^r f$ agree and are algebraic on CM points. In particular, we have the following:
\begin{align}
\label{eqn:deltafinK}
 d^{r} f(A, A[\nn], \Omega_A) = \delta^{r} f(A, A[\nn], \omega_A)   \in E_fK.
\end{align}

We are ready to state Waldspurger's formula \cite[Theorem~5.4]{BDP1}.
\begin{theorem}[Waldspurger's formula]
\label{thm:waldspurger}
\begin{align*}
(\delta^{r -1} f(A, A[\nn], \omega_A))^2 = 1/2 (2 \pi / \sqrt{D})^{2r -1} (r-1)! r!  \frac{L(f, \chi_r , 0)}{(2 \pi i \overline{\alpha}\Omega_K)^{4r}}.
\end{align*}
\end{theorem}

 To interpolate $L(f, \chi_r , 0)$ and therefore \[(\delta^{r -1} f(A, A[\nn], \omega_A))^2 = (d^{r -1} f(A, A[\nn], \omega_A))^2\] $p$-adically, we take the $p$-depletion
\begin{align*}  
d^{r-1} f^{[p]} (q) \colonequals \sum_{(n, p) = 1} n^{r-1} a_n(f) q^n .
\end{align*}
The set $\{ d^{r-1} f^{[p]} \}$ is a $p$-adic family of modular forms, and there exists a $p$-adic period $\Omega_p \in \C_p^\times$ such that
\begin{align*} L_p(f, \chi_r)  \colonequals \Omega_p^{4r}( d^{r-1}f^{[p]}(A, A[\nn], \omega_A)) ^2
\end{align*}
extends to a $p$-adic analytic function of $r \in (\Z/(p-1)\Z) \times \Z_p$ \cite[Theorem~5.9]{BDP1}.

\begin{remark}
The period pair $(\Omega_K, \Omega_p)$ is only well-defined as a pair: both periods depend linearly on the choice of $\omega_A$, but their ratio is independent of this choice of scalar multiple.
\end{remark}

To compute the anticyclotomic $p$-adic $L$-function in the range of interpolation, we need to compute the Shimura--Maass derivative $\delta^{r-1} f (\tau_\nn)$.
Let $\chi_r \in S$ be the Galois character associated to the Hecke character of $K$ of infinity type $(1 + r, 1-r)$ with $r \geq 1$. Then 
\begin{align}
\label{eqn:Lalg}
L_{\alg}(f, \chi_r^{-1}, 0) = \delta^{r -1} f(A, A[\nn], \omega_A)^2 = \frac{(\delta^{r-1} f(\tau_\nn))^2}{ (\bar{\alpha} \Omega_K )^{4r}} .
\end{align}

We would like to evaluate the right hand side of \eqref{eqn:Lalg}.
As $r$ gets large in the usual absolute value, this value also gets large, and a naive strategy like applying the equality \eqref{eqn:deltafinK} and evaluating a truncated $q$-expansion of $d^{r-1} f(\tau_\nn)$ does not approximate the true algebraic value well.

However, \cite[Section~6.3]{Zagier123} and \cite{ZagierVillegas} show that the values of the Shimura--Maass derivative of a modular form evaluated at a CM point satisfy a recurrence relation due to a large amount of algebraic structure on the ring of modular forms $M_*(\Gamma_0(N))$. We recall briefly some of the essential ideas involved in the proof.

We introduce the $\vartheta$ differential operator that acts on a weight $k$ modular form $g$ by
\begin{align*} \vartheta g = dg - \frac{k}{12} E_2 g \end{align*}
where $E_2$ is the weight $2$ Eisenstein series.
We modify the $\vartheta$ operator by the following recursive definition. Let $\vartheta^{[0]}g = g$ and 
\begin{align}
\vartheta^{[r+1]} g  = \vartheta ( \vartheta^{[r]} g) - r (k+r-1) (E_4/144) \vartheta^{[r-1]} g \text{ for } r \geq 1.
\end{align}

The \defi{Cohen--Kuznetsov series} are formal generating series attached to the differential operators that have nice transformation properties under $\Gamma_0(N)$. We provide the details necessary for our calculations here; the interested reader can see \cite[Section~5.2]{Zagier123} or \cite{ZagierProc} for more details, as well as \cite[Section~7]{ZagierVillegas} for some example calculations. The relationships between the Cohen--Kuznetsov series show that the set of values $\{\delta^r(\tau_\nn)\}_{r \geq 0}$ inherits a recursive relation coming from the recursive definition of $\vartheta^{[i]}$.
We can define a Cohen--Kuznetsov series associated to $\vartheta$ and $\delta$ by
\begin{align*}
\tilde{g}_\vartheta(z, X) = \sum_{n = 0}^\infty \frac{\vartheta^{[n]}g(z)}{n!(k)_n} X^n \text{ and }
 \tilde{g}_\delta(z, X) = \sum_{n=0}^\infty \frac{ \delta^n g(z)}{n! (k)_n} X^n.
\end{align*}
These power series satisfy
\begin{align}
\label{eqn:cohenkuzreln}
 \tilde{g}_\vartheta(z, X) =  e^{-X E_2^*(z)/12} \tilde{g}_\delta(z, X)
 \end{align}
where $E_2^*(z)\colonequals E_2(z) - 3/(\pi y)$.

The key idea is that if $E_2^*(z_0) = 0$, then \eqref{eqn:cohenkuzreln} implies that $\vartheta^{[i]}g(z_0) = \delta^i g(z_0)$ for all $i \geq 0$. Then, because $\vartheta^{[i]}g(z_0)$ is defined recursively, we can compute $\delta^i g(z_0)$ recursively. This method requires two pieces of input. 
\begin{enumerate}
\item We modify the operator $\vartheta^{[i]}$ for the CM point $\tau_\nn$ with an appropriate holomorphic function $\phi(z)$ such that $\phi^*(z)\colonequals \phi(z) - 1/(4 \pi y)$  transforms like a modular form of weight $2$ on $\Gamma_0(N)$. Define $\vartheta_{\phi} g\colonequals  d g - k \phi g$. The resulting relationship \eqref{eqn:cohenkuzreln} becomes $ \tilde{g}_{\vartheta_{\phi}}(z, X)  = e^{-X \phi^*(z)} \tilde{g}_\delta(z, X)$  on $\calH$, and we require $\phi^*(\tau_\nn) = 0$.
\item We need generators $g_1, \dots, g_n$ for $M_*(\Gamma_0(N))$ so that we can compute $\vartheta_\phi$ of each generator as well as the values $g_i (\tau_\nn)$ and $(\vartheta_\phi g_i) (\tau _\nn)$ for each generator to determine the recurrence relation.
\end{enumerate}

We can always take $\phi = (1/12) E_2 +A$ for some weight $2$ holomorphic or meromorphic function $A$ on $\Gamma_0(N)$ \cite[Section~5.2]{Zagier123}. Therefore to compute $\phi$ we evaluate both $E_2$ and a basis of weight $2$ level $N$ forms on $\tau_\nn$. Then we solve for $\phi$ by finding a linear combination of the weight $2$ level $N$ forms that when evaluated at $\tau_\nn$ equal to $E_2$ evaluated at $\tau_\nn$.

Define the modular form $\Phi \colonequals  d \phi - \phi^2$. We also obtain operators $\vartheta_{\phi}^{[n]}$ satisfying the recurrence relation 
\begin{align} \vartheta_{\phi}^{[0]} g = g, \vartheta_{\phi}^{[r+1]} g =  \vartheta_{\phi} ( \vartheta_{\phi}^{[r]} g) + r (k+r-1) \Phi (\vartheta_{\phi}^{[r-1]} g).
\end{align}

To rigorously evaluate $g_i (\tau_\nn)$ and  $(\vartheta_\phi g_i) (\tau_\nn)$ as algebraic numbers we need to bound the denominators of these values that, a priori, lie in $K$ by \eqref{eqn:deltafinK}.
With work, one can explicitly bound these denominators. Let $t \colonequals (p -1)^{-1}$. By \cite[Theorem~B.3.2.1]{BDPconiveau}, when $\ord_p(N) = 1, 2,$ or $3$, the denominators are  at most $p^{\lceil k (t+1) \rceil}$, $p^{\lceil 3kt (2t+1) \rceil}$, or $p^{\lceil 2kt (t+1)(t+2) \rceil}$ respectively. 

We give some intuition following \cite[Appendix~B]{BDPconiveau}.
Consider the $\calO_K$-module $M_k(\Gamma_0(N), \calO_K)$ consisting of weight $k$ modular forms $g$ for $\Gamma_0(N)$ whose $q$-expansions have coefficients in $\calO_K$. For $g \in M_k(\Gamma_0(N), \calO_K)$, the evaluation of $g(A, A[\nn], \omega_A)$ belongs to $\calO_K[1/N]$.
One can construct a finite index $\calO_K$-module $M_{k, \calO_K}$ of the module $M_k(\Gamma_0(N), \calO_K)$ such{ that for every $g \in M_{k, \calO_K}$, the evaluation of $g(A, A[\nn], \omega_A)$ is $\calO_K$-integral. This is tensor-compatible with the corresponding $\Z$-modules:
\[ \calO_K \otimes_\Z M_{k, \Z} = M_{k, \calO_K} \otimes_\Z M_k (\Gamma_0(N), \Z)  = M_k(\Gamma_0(N), \calO_K)\]
so the exponent of the finite abelian group $M_k (\Gamma, \Z) / M_{k, \Z}$ multiplies $M_k(\Gamma_0(N), \calO_K)$ into $M_{k, \calO_K}$
and is an explicit bound on the denominator. (In fact, it gives a much stronger result: this gives a bound on all denominators for all number fields.)

\begin{example}
\label{ex37a1}
Let $f$ be the modular form  \href{https://www.lmfdb.org/ModularForm/GL2/Q/holomorphic/37/2/a/a/}{\texttt{37.2.a.a}} with weight $2$ and level $37$. We have a basis of weight $2$ forms on $\Gamma_0(37)$ given by
\begin{align*}
f_1 &= 1 - 2q^3 + 10q^4 + 2q^5 + 14q^6 + 6q^7 + 10q^8 + 18q^9 + O(q^{10}) 
    \notag\\
f_2 &= q + q^3 - 2q^4 - q^7 - 2q^9 +O(q^{10})\\
f_3 &= q^2 + 2q^3 - 2q^4 + q^5 - 3q^6 - 4q^9 + O(q^{10}).\notag
\end{align*}
Let $\tau_\nn = \frac{-27 + \sqrt{-11}}{2 \cdot 37}$. Let $p = 5$.
By the discussion above, we can bound the denominators of $f_i(\tau_\nn)$  and $E_2^*(\tau_\nn)$ by $37^{\lceil 2 p / (p-1) \rceil} = 37^3$.
We can compute $E_2^*(\tau_\nn)/(\Omega_A)^2 = 4400 -3696 \sqrt{-11}$. We also compute $f_i(\tau_\nn)$:
\begin{align*}
f_1(\tau_\nn)/\Omega_A^{2} &= 2420 +572 \sqrt{-11}\notag\\
f_2(\tau_\nn)/ \Omega_A^{2} &=-726 +154 \sqrt{-11} \\
f_3(\tau_\nn)/\Omega_A^{2} &= -1210 -286 \sqrt{-11}.\notag
\end{align*}
Therefore $\phi = 1/12 E_2 -1/12(-28/11 f_1 -160/11 f_2)$.
\end{example}

Given an expression $M_*(\Gamma_0(N)) \simeq \Q[g_1, \dots,  g_n]/I$, we can use the iterative relation
\begin{align}
\vartheta_{\phi}^{[r+1]} f =  \vartheta_{\phi} ( \vartheta_{\phi}^{[r]} f) + r (r +1) \Phi (\vartheta_{\phi}^{[r-1]} f) \text{ for } r \geq 1
\end{align}
and apply the Leibniz rule to the monomials in $\vartheta_{\phi}^{[r]} f$ to apply $\vartheta_\phi$ iteratively.
By \cite[Corollary 1.5.1]{voight2019canonical} to obtain generators and relations for $ M_*(\Gamma_0(N))$  we need generators up to degree $6$ and relations up to degree $12$.

\begin{example}
Continuing Example \ref{ex37a1}, we can use these methods to compute
\begin{align*}
 f (\tau_\nn)/\Omega_A^{2} = 1694 + 726 \sqrt{-11} \text{ and } \delta f (\tau_\nn) /\Omega_A^{4} = 532400 - 447216 \sqrt{-11}.
 \end{align*}
 \end{example}
Note we have fixed an embedding $K \to \Q_p$ given by the splitting $p = \pp \overline{\pp}$. We define $\pp$ to be the prime with valuation $1$. Then $K_\pp \simeq \Q_p$.

\begin{remark}
\label{rem:embeddings}
When $[E_f:\Q] > 1$, then  $\delta^{r-1} f (\tau_\nn) / \Omega_A^{2r}$ belong to the compositum of $K$ and $E_f$ and we embed $e: E_f \to \Q_p$. Changing the embedding $e$ is equivalent to picking a Galois conjugate $f^\sigma$ for $\sigma \in \Gal(E_f/\Q)$.  Since the operator $d$ is Galois-equivariant, and we have the equality \eqref{eqn:deltafinK}, it follows\[\delta^{r-1} f^\sigma(\tau_n) = \sigma (\delta^{r-1} f(\tau_n)).\]
In this way, we can obtain all $[E_f:\Q]$ $p$-adic values of $\delta^{r-1} f^\sigma(\tau_n)$ for $\sigma \in \Gal(E_f/\Q)$ from knowing a single value $\delta^{r-1} f(\tau_n)$.
\end{remark}

\subsection{Evaluating outside of the range of interpolation}
\label{sec:outrangeinterp}

We now discuss an adaptation of Rubin's method to compute the value $L_p(f,1)$ outside of the range of interpolation.

Let $r \in \N$. Let $\chi_r  \in S$ be the Galois character associated to the Hecke character with infinity type $(1 + r, 1 -r)$. Define
\begin{align}
\label{def:ell}
\ell(r) \colonequals L_p(f,  \chi_r) \Omega_p^{- 4 r}.
\end{align}
We want to compute $\ell(0)$. Since this is not in the range of interpolation, we compute auxiliary values  $\ell((p-1)), \ell(2(p-1)), \dots, \ell(B (p-1)) $ in the range of interpolation and recover $\ell(0)$ modulo $\pp^{B}$ from a version of \cite[Theorem~9, Proposition~7]{RubinMonodromy} for the anticyclotomic $p$-adic $L$-function of Bertolini, Darmon, and Prasanna.

The main result of this section is the following proposition.
\begin{proposition}
\label{MainProp}
Let $\ell(r)$ be defined as above. Then for any $B \in \N$, we have
\begin{align} \ell(0)^{(p-1)/2}  \equiv \sum_{j=1}^B \left( \sum_{i = j}^B (-1)^{j-1} \binom{i-1}{j-1} \right) \ell(j(p-1))^{(p-1)/2} \mod \pp^B.\end{align}
Furthermore, $\ell(0) \equiv \ell((p-1)^2/2) \mod \pp$.
\end{proposition}

Assuming $\ell(0) \not\equiv 0 \mod \pp$, Proposition \ref{MainProp} allows us to uniquely recover $\ell(0)$ from the auxiliary values $\ell(j (p-1))$. We now prove Proposition \ref{MainProp}.

Following Rubin, we introduce a ring $\calI$ of generalized Iwasawa functions. A function $g$ on $\Z_p$ is in $\calI$ if there exist units $u_1, \dots, u_m \in 1+ \pp \hat{\calO}$ and a power series $H \in \hat{\calO}[[X_1, \dots, X_m]]$ such that $g(s) = H(u_1^s - 1, \dots, u_m^s-1)$ for all $s \in \Z_p$.

Recall $\chi_{i(p-1)}  \in S$ is a Galois character associated to a Hecke character of $K$ with infinity type $(1 + i(p-1), 1 -i (p -1))$. By composing with a projection arising from $\Z_p^\times \simeq (\Z/p\Z)^\times \times (1+ p \Z_p)$, we have
\begin{align} 
\label{eqn:chisubi}
\langle \chi_{i(p-1)} \rangle : \Gamma^- \to \Z_p^\times \to 1+ p \Z_p
\end{align} since $\chi_{i(p-1)}$ already takes values in $1+ p \Z_p$ \cite[\S II.4.17]{deShalitCMIwasawa}. For $F \in \Lambda^\ac$, we have $\langle \chi_{i(p-1)} \rangle(F) \in \widehat{\overline{\Q}}_p$, and furthermore for $s \in \Z_p$ we can define $\chi_{s(p-1)} \colonequals \langle\chi_{(p-1)} \rangle^s$ and evaluate  $ \chi_{s(p-1)} (F)\in \widehat{\overline{\Q}}_p$ by continuity.

Define
\begin{align}
\label{def:H}
H(s) \colonequals (\Omega_p^{(p-1)^2})^{-2s} L_p(f, \chi_{s(p-1)})^{(p-1)/2}.
\end{align}
By \cite[\S II.4.3(10)]{deShalitCMIwasawa}, we have
\begin{align}
\label{omegapcong}
\Omega_p^{(p-1)^2} \in 1 + \pp \hat{\calO}
\end{align}
so $H$ is well-defined.
For $i \in \Z$, note that 
\begin{align}
H(i) = \ell((p-1) i )^{(p-1)/2}.
\end{align}

Analogously to \cite[Proposition~7]{RubinMonodromy}, we have the following proposition.
\begin{proposition}\label{Prop7} Let $\ell(r)$ and $H$ be defined as in \eqref{def:ell} and \eqref{def:H}. Then
\begin{enumerate}
\item $H \in \calI$.
\item \label{congruence} $\ell(0) \equiv \ell((p-1)^2/2) \mod \pp$.
\end{enumerate}
\end{proposition}
\begin{proof}
We define functions $f_1$ and $f_2$ by
\begin{align}
f_1(s) &\colonequals L_p(f, \chi_{s(p-1)})\\
f_2(s) &\colonequals (\Omega_p^{(p-1)^2})^{-2s}. \notag
\end{align}
Then $H = f_1^{(p-1)/2} f_2$. We know $f_1 \in \calI$ since $\chi_{i(p-1)}$ is a character into $1+ p \Z_p$ for all $i \in \N$ (see \eqref{eqn:chisubi}) and by \eqref{omegapcong} we know $f_2 \in \calI$, so $H \in \calI$.

Finally, since $f_1(s) \equiv f_1(s') \mod \pp$ for all $s, s'\in \Z_p$ we have
\begin{align}
\ell(0) = f_1(0) \equiv \Omega_p^{-2 (p-1)^2} f_1((p-1)/2) = \ell((p-1)^2/2) \mod \pp.
\end{align}
\end{proof}

\begin{remark}
\label{rem:highercongruences} Proposition \ref{Prop7} \eqref{congruence} is only helpful if $ \ell((p-1)^2/2) \not\equiv 0 \mod \pp$. More generally, one can see that by \eqref{omegapcong}, for $n \geq 1$ we have
\begin{align}
\ell(0) = f_1(0) \equiv \Omega_p^{-2 (p-1)^2 p^{n-1}} f_1((p-1)p^{n-1}/2) = \ell((p-1)^2 p^{n-1}/2) \mod \pp^n.
\end{align}
The main difficulty in applying this congruence is computing $\ell((p-1)^2 p^{n-1}/2)$.
\end{remark}

From here it is straightforward to follow Rubin's proof to obtain a proof of Proposition~\ref{MainProp}: we give a brief summary. He defines a difference operator $\Delta$ on $\calI$ by $\Delta(g)(s) \colonequals g(s+1) - g(s)$. If $g \in \calI$ then $\Delta(g) \in \pp \calI$ \cite[Lemma~8]{RubinMonodromy}.

By inverting $(1 + \Delta)^{-1} = \sum_{i=0}^\infty (-1)^i \Delta^i$ and applying the congruence, we obtain the desired formula \cite[Theorem~9]{RubinMonodromy}
\begin{align}
\label{h0formula}
g(0) = \sum_{j = 1}^B  \left( \sum_{i =j}^B (-1)^{j-1} \binom{i-1}{j-1} \right) g(j) \mod \pp^B.
\end{align}
By applying this to $H \in \calI$ we can compute the special values.

\begin{table}[!ht]
\begin{center}
\begin{tabularx}{268pt}{ccrrc}
$f$ & $p$ & $D$ & time & Sturm Bound \\
\midrule
\href{https://www.lmfdb.org/ModularForm/GL2/Q/holomorphic/37/2/a/a/}{\texttt{37.2.a.a}} & $5$ & $-11$ & $13.970$ & $7$ \\
 \href{https://www.lmfdb.org/ModularForm/GL2/Q/holomorphic/43/2/a/a/}{\texttt{43.2.a.a}} & $5$ & $-19$ &  $18.480$ & $8$\\
 \href{https://www.lmfdb.org/ModularForm/GL2/Q/holomorphic/58/2/a/a/}{\texttt{58.2.a.a}} & $11$ & $-7$ &$1583.380$  &  $15$\\
  \href{https://www.lmfdb.org/ModularForm/GL2/Q/holomorphic/61/2/a/a/}{\texttt{61.2.a.a}} & $5$ & $-19$ & $34.240$ &  $11$\\
 \href{https://www.lmfdb.org/ModularForm/GL2/Q/holomorphic/83/2/a/a/}{\texttt{83.2.a.a}}& $5$ &$-19$ & $73.400$ & $14$\\
 \href{https://www.lmfdb.org/ModularForm/GL2/Q/holomorphic/89/2/a/a/}{\texttt{89.2.a.a}} & $3$ &$-11$ &$15.730$ & $15$ \\
\href{https://www.lmfdb.org/ModularForm/GL2/Q/holomorphic/77/2/a/a/}{\texttt{77.2.a.a}}& $5$ & $-19$ & $65.350$& $16$\\
  \href{https://www.lmfdb.org/ModularForm/GL2/Q/holomorphic/101/2/a/a/}{\texttt{101.2.a.a}}&$5$ & $-19$ & $95.850$ & $17$ \\
 \href{https://www.lmfdb.org/ModularForm/GL2/Q/holomorphic/131/2/a/a/}{\texttt{131.2.a.a}}& $5$ & $-19$& $326.160$ & $22$
\end{tabularx}
\end{center}
\caption{Timings for fixed $B = 5$ and varying $N$ for modular forms with rational Fourier coefficients} 
\label{tableoftimingsfixedB}
\end{table}

\begin{remark}
We have written code to compute $\ell(0)$ and $L_p(f, 1)$ in Magma V2.26-11. It can be found at \cite{GitRepoChabautyGZ}. 
See Table \ref{tableoftimingsfixedB} for some timings. These timings were done on a 2017 Macbook Pro with a 2.3 GHz Dual-Core Intel Core i5 processor and 8 GB of RAM. All times are given in seconds.
\end{remark}

\begin{table}[h!]

\begin{subtable}[h]{0.45\textwidth}
\begin{center}
\begin{tabularx}{120pt}{cr}
$r$ & $\ell(r) \mod \pp^{10}$\\
\midrule
$4$  & $-2341944$\\
$8$  & $830906$\\
$12$ & $-3933069$\\
$16$ &$-35494$ \\
$20$ & $1760756$ \\
$24$ & $1706556$\\
$28$ & $1972781$\\
$32$ & $-3662194$\\
$36$ & $3734381$\\
$40$ &$ 4015256$
\end{tabularx}
\end{center}
\caption{$\ell(r)$ for \href{https://www.lmfdb.org/ModularForm/GL2/Q/holomorphic/37/2/a/a/}{\texttt{37.2.a.a}}} 
\label{tableofvalues37}
\end{subtable}
\hfill
\begin{subtable}[h]{0.45\textwidth}
\begin{center}
\begin{tabularx}{120pt}{cr}
$r$ & $\ell(r) \mod \pp^{10}$\\
\midrule
$6$  & $19467645$\\
$12$  & $27057027$\\
$18$ & $-443168$\\
$24$ &$72608418$ \\
$30$ & $-32171562$ \\
$36$ & $-95344303$\\
$42$ & $-68492493$\\
$48$ & $-129070518$\\
$54$ & $-81233717$\\
$60$ &$ 24574313$
\end{tabularx}
\caption{$\ell(r)$ for $f$ in \href{https://www.lmfdb.org/ModularForm/GL2/Q/holomorphic/85/2/a/b/}{\texttt{85.2.a.b}}}
\label{tableofvalues85f}
\end{center}
\end{subtable}
\caption{Example computations of $\ell(r)$}
\end{table}

\begin{example}
\label{ex:log37}
Let $f$ be the newform with LMFDB label \href{https://www.lmfdb.org/ModularForm/GL2/Q/holomorphic/37/2/a/a/}{\texttt{37.2.a.a}}, $D = -11$, and $p = 5$. We compute the values of $\ell(r)$ in Table \ref{tableofvalues37}.
So Proposition~\ref{MainProp} implies that 
\begin{align*}
H(0) = L_p(f, 1)^2   \equiv  2502536 \mod \pp^{10}
\end{align*}
and 
\begin{align*}
 L_p(f, 1)\equiv \ell(8) \equiv 830906 \mod \pp
\end{align*}
so 
\begin{align*}
 L_p(f, 1) \equiv 4635631 \mod \pp^{10}.
\end{align*}
\end{example}

\begin{example}
For \href{https://www.lmfdb.org/ModularForm/GL2/Q/holomorphic/77/2/a/a/}{\texttt{77.2.a.a}}, $D = -19$, and $p = 5$,  we can similarly compute $\ell(0) \mod B= 7$:
\begin{align*}
  L_p(f, 1) \equiv 4 + 2\cdot5 + 4\cdot5^2 + 3\cdot5^3 + 5^4 + 3\cdot5^6  \mod \pp^{7}.
\end{align*}
This agrees with the computed value of $\left( \frac{1 - a_p(f) +p}{p}\right)^{2}\log(\pi(y_K))^2$ on the associated elliptic curve $E$. In this case $\pi_E(y_K)$ has index $2$ in $E(\Q)$.
\end{example}

\begin{example}
\label{ex:log85}
Let $f$ be a newform in the orbit \href{https://www.lmfdb.org/ModularForm/GL2/Q/holomorphic/85/2/a/b/}{\texttt{85.2.a.b}}. Then $f$ has coefficient field $\Q(\sqrt{2})$, and we denote
\begin{align*}
f(q) &\colonequals q + (\sqrt{2} - 1)q^2 + (-\sqrt{2} - 2)q^3 + (-2\sqrt{2} + 1)q^4 - q^5 - \sqrt{2}q^6 + O(q^7)\end{align*}
Let $D = -19$, and $p = 7$ and fix the embedding $\sqrt{2}\mapsto 4 + 5\cdot 7 + 4\cdot 7^2 + O(7^4)$.
We compute the values of $\ell(r)$ in Table \ref{tableofvalues85f}.

So Proposition~\ref{MainProp} implies that 
\begin{align*}
 L_p(f, 1) \equiv -25026440\cdot 7^2 \mod \pp^{10}.
\end{align*}
Per Remark \ref{rem:embeddings}, we can use the same algebraic values used to compute Table \ref{tableofvalues85f} under the other $p$-adic embedding to compute
\begin{align*}
 L_p(f^\sigma, 1) \equiv -107584760\cdot 7^{-2} \mod \pp^{10},
\end{align*}
where $\sigma \in \Gal(\Q(\sqrt{2})/\Q)$ is the nontrivial automorphism.

In other words, let $X_0(85)^* \colonequals X_0(85) / \langle w_5, w_{17} \rangle$ be the full  Atkin--Lehner quotient of $X_0(85)$ with Jacobian $J_0(85)^*$ and $\pi:J_0(85) \to J_0(85)^*$.
We have determined $(\log_{f dq/q} \pi(y_K))^2$ and $(\log_{f^\sigma dq/q} \pi(y_K))^2$, the values of the square of the logarithm of the $\pi(y_K)$ on the Jacobian of $X_0(85)^*$ with respect to the basis $f dq/q, f^\sigma dq/q$ of $H^0(X_0(85)^*, \Omega^1)$.
\end{example}

\begin{table}[h!]
\begin{center}
\begin{tabularx}{260pt}{cl}
$r$ & $\ell(r) \mod \pp^{9}$\\
\midrule
$10$ & $-297386117\cdot \nu^2 + 592900508\cdot \nu + 979493595$\\
$20$ & $ 476184186 \cdot \nu ^2 - 623749256\cdot \nu - 274455717 $\\
$30$ & $ -716852945\cdot \nu^2 + 898090404\cdot \nu - 230653046 $\\
$40$ & $ -813805284\cdot \nu^2 +944452192 \cdot \nu- 205004852$\\
$50$ & $652023685\cdot \nu^2 + 765243602\cdot \nu + 477480226$ \\
$60$ & $-730757902\cdot \nu^2 + 146316051\cdot \nu - 343684998$ \\
$70$ & $-529440885\cdot \nu^2  - 683691281\cdot \nu +99641402$\\
$80$ & $-834057296\cdot \nu^2+ 568271068\cdot \nu - 580138952$\\
$90$ & $84670657\cdot \nu^2+476068012\cdot \nu +  511028466 $\\
\end{tabularx}
\end{center}
\caption{$\ell(r)$ for \href{https://www.lmfdb.org/ModularForm/GL2/Q/holomorphic/169/2/a/b/}{\texttt{169.2.a.b}}} 
\label{tableofvalues169}
\end{table}

\begin{example}

Let $f$ be a newform in the orbit \href{https://www.lmfdb.org/ModularForm/GL2/Q/holomorphic/169/2/a/b/}{\texttt{169.2.a.b}}. Then $f$ has coefficient field $E_f = \Q(\zeta_{14})^+$, and we denote
\begin{align*}
f(q) &\colonequals q + (-\nu^2 + 1)q^2 + (\nu^2 - \nu - 2)q^3 + (\nu^2 + \nu - 2)q^4 + (-\nu^2 + \nu)q^5 +  O(q^6)\end{align*} where $\nu$ satisfies the minimal polynomial $x^3 - x^2 - 2x +1$.
Let $D = -43$, and $p = 11$. Then $p$ is inert in $E_f$, so does not satisfy one of our assumptions. 
However, it is still possible to compute $L_p(f, 1)$ by working in an extension of $\Q_p$.
The $\mathfrak{p}$-adic completion of the compositum $KE_f$ is a degree 3 unramified extension of $\Q_p$. Let $\nu \mapsto 16978\cdot \nu ^2 + 53324\cdot \nu - 31046 + O(11^5)$ be the embedding of $E_f$ into this extension. 

We compute the values of $\ell(r)$ in Table \ref{tableofvalues169}. Then
Proposition~\ref{MainProp} implies that 
\begin{align*}
 L_p(f, 1) \equiv  -1049872412\cdot \nu^2 - 976527363\cdot \nu + 889741537     \mod \pp^9.
\end{align*}
By considering the other two embeddings from $E_f$ into the $\mathfrak{p}$-adic completion of $KE_f$, we obtain the values of $L_p(f^\sigma, 1)$ for $\sigma \in \Gal(E_f/\Q)$. This allows us to compute $(\log_{f^\sigma dq/q} y_K)^2$ for the genus $3$ modular curve $X_0(169)^+$ which is isomorphic over $\Q$ to the split Cartan modular curve $X_{\mathrm{s}}(13)$.
\end{example}

\section{Perrin-Riou's \texorpdfstring{$p$}{p}-adic Gross--Zagier formula}
\label{Ch:Height}

Recall that $f \in S_2(N)$ is a weight $2$ newform for $\Gamma_0(N)$.
In this section we discuss the computation of the cyclotomic $p$-adic height of the $f$-isotypical component of the Heegner point using the $p$-adic Gross--Zagier formula of Perrin-Riou \cite{PerrinRiouInvent}. Like the anticyclotomic $L$-function of the previous section, Perrin-Riou's $p$-adic $L$-function $\calL_p(f)$ also interpolates the central value of the Rankin $L$-series $L(f, \chi, 1)$ for certain Hecke characters $\chi$, and she also provides a $p$-adic Gross--Zagier formula with a special value at $\mathbf{1}$. The nature of Perrin-Riou's construction of the $p$-adic Rankin $L$-series associated to $f$ and $K$ leads $\mathbf{1}(\calL_p(f))$ to vanish, and instead she considers the derivative in the cyclotomic direction \eqref{defn:derivcyclo}, which she shows is proportional to the $p$-adic height of the Heegner point.

We outline how to compute the derivative $\calL_p(f)$ at $\mathbf{1}$ in the cyclotomic direction, as well as the other constants appearing in the $p$-adic Gross--Zagier formula, and therefore the height of the Heegner point. 
In order to do this, we will use the relationship between Perrin-Riou's $p$-adic $L$-function and the $p$-adic $L$-function of Amice--V\'elu and Vishik. We then use the theory of overconvergent modular symbols \cite{PSOverconvergent} to compute values of the latter $p$-adic $L$-function.

For this section we have the following assumptions.
\begin{assumption}\label{hyp:heights}
\hfill
\begin{enumerate}
\item Assume $p>2$ is a prime number not dividing $N$;
\item assume $p$ splits in $E_f$, the coefficient field of $f$, and fix an embedding $e:E_f \to \Q_p$;\label{hyp:Efsplitting2}
\item assume $f^\sigma$ is ordinary in $e$, that is, $e(a_p(f^\sigma))$ is a unit for all $\sigma \in \Gal(E_f/\Q)$; \label{hyp:ordinaryp}
\item assume $f$ has analytic rank $1$;
\item let $K = \Q(\sqrt{D})$ be an imaginary quadratic field of class number $1$ with odd discriminant $D < - 3$; \label{hyp:classnumber1v2}
\item assume $p$ splits in $K$; \label{hyp:psplitK}
\item assuem every prime $q$ dividing $N$ splits in $K$.
\end{enumerate}
\end{assumption}
Like in the previous section, Assumptions \eqref{hyp:Efsplitting2}, \eqref{hyp:psplitK}, and \eqref{hyp:classnumber1v2} are simplifying assumptions to avoid working in field extensions of $\Q_p$. We require Assumption \eqref{hyp:ordinaryp} for the construction of the $p$-adic $L$-function: \cite{Kobayashi} constructs an analogous $p$-adic $L$-function in the case of supersingular reduction.

Let $K_\infty / K$ be the (unique) $\Z_p^2$-extension of $K$ with Galois group $\Gamma \colonequals \Gal(K_\infty /K)$. Write $\calO$ for the ring of integers in $\overline{\Q}_p$.  Let $\psi: \Gamma \to \calO^\times$ be a finite order character. The theta series
\begin{align}
\Theta_\psi \colonequals  \sum_{\aa \subset \calO_K}\psi(\aa) q^{\mathrm{Nm}(\aa)}
\end{align}
is a weight $1$ modular form. We denote the Rankin--Selberg convolution by $L(f, \psi, s) \colonequals L(f, \Theta_\psi, s)$.
Perrin-Riou constructs a $p$-adic $L$-function $\calL_p(f) \in \Z_p[[\Gamma ]]$ characterized by the interpolation property that for any finite order character $\psi: \Gamma \to \calO^\times$, the values
\begin{align}
\label{eqn:PRinterpolation}
\psi(\calL_p(f)) \doteq L(f,  \psi , 1)
\end{align}
are proportional. (As in \eqref{eqn:BDPinterpolation}, the left hand side is our notation for evaluation $\calL_p(f)$ at $\psi$ in the range of interpolation.)
Using the functional equation for $\calL_p(f)$ we can show that that $\calL_p(f)$ vanishes at the trivial character $\mathbf{1}$. 

Perrin-Riou relates the derivative of this $p$-adic $L$-function at $\mathbf{1}$ to a $p$-adic height pairing defined by Schneider and Mazur--Tate \cite{Schneider,MazurTateBiext} on abelian varieties. This coincides also with the height pairing of Coleman--Gross \cite{ColemanGross} in the case of Jacobians, when the required splitting of the Hodge filtration is given by the unit root subspace of $H^1(X_{\Q_p}, \Omega^1)$ \cite{Colemanbiext}.

This height pairing $\langle \cdot, \cdot \rangle_{\ell_K}$ depends on a choice of id\`ele class character:
\begin{align*}
\ell: \A_K^\times/K^\times \to \Q_p
\end{align*}
equivalently, by class field theory, a homomorphism $\ell: \Gal(K_\infty / K) \to \Q_p$, which we will take to be the cyclotomic character.  Let $\ell_K$ be the restriction of $\ell$ to $\Gal(\overline{\Q}/K)$.
We can decompose $\ell$ as the composition $\ell = \log_p \circ \lambda$ where $\lambda : \Gamma \to 1+ p \Z_p$ is the cyclotomic character.

The derivative of $\calL_p(f)$ in the direction of $\ell$ at the trivial character $\mathbf{1}$ is defined as 
\begin{align}
\label{defn:derivcyclo}
\calL'_{p, \ell}(f, 1) \colonequals  \left(\frac{d}{ds} \lambda^s(\calL_p (f))\right)\bigg|_{s = 0}.
\end{align}

Recall that $y_K \in J_0(N)(K)$ denotes the Heegner point associated to $K$, and for $\sigma \in \Gal(E_f/K)$, we write $y_{K, f^\sigma}$ for the $f^\sigma$-isotypical part of $y_K$, which belongs to $J_0(N)(K)\otimes E_f$.  

By Assumption \ref{hyp:heights} \eqref{hyp:ordinaryp}, $f$ is ordinary in $e$; let $\alpha_p(f)$ denote the unit root of the Frobenius polynomial $x^2 - e(a_p)x + p$ in $\Q_p$.
\begin{theorem}[{\cite[Theorem 1.3]{PerrinRiouInvent}}]
\label{thm:PRpadicGZ}
The function $\calL_p(f)$ vanishes at $\mathbf{1}$ and the derivative at $\mathbf{1}$ in the direction $\ell$ is
\begin{align}
\calL'_{p, \ell}(f, 1) =  \left( 1 - \frac{1}{\alpha_p(f)} \right)^{4} \langle y_{K,f}, y_{K,f} \rangle_{\ell_K}.
\end{align}
\end{theorem}

We now discuss how to modify Theorem \ref{thm:PRpadicGZ} to obtain the heights of the images of Heegner points on quotients $J_X$ of $J_0(N)$.
These formulas lay the groundwork for doing quadratic Chabauty on quotients of $X_0(N)$.

\begin{proposition}[{\cite[(3.4.3)]{MazurTateBiext}}]
\label{prop:ProjectionFormula}
Let $g: A \to B$ be a homomorphism of principally polarized abelian varieties over $\Q$, $a \in A(K)$, and $b  \in B(K)$. 
Let $g^\vee$ denote the dual map 
$g^{\vee}: B^\vee \to A^\vee$, while $\lambda_A: A \to A^\vee$ and  
$\lambda_B: B \to B^\vee$ denote the principal polarizations.
Then
\[ \langle a,(\lambda_A^{-1} \circ g^\vee \circ  \lambda_B ) (b)\rangle_{\ell_K}  = \langle g (a), b \rangle_{\ell_K} .\]
\end{proposition}
We deduce the following proposition from Mazur and Tate's formula.


\begin{proposition}
\label{cor:PRpadicGZ}
Let $\phi:X_0(N) \to X$ be a simple new $\Gamma_0(N)$-modular curve, with Jacobian $J_X$ and an associated newform $f_X$.
Let $\pi: J_0(N) \to J_X$ be induced by $\phi_*$.
We continue to suppose $p$ is ordinary for $J_X$ (see Assumption \ref{hyp:heights} \eqref{hyp:ordinaryp}). Then
\begin{align}
\calL'_{p, \ell}(f_X, 1) =\left( 1 - \frac{1}{\alpha_p(f_X)} \right)^{4} \frac{\langle \pi(y_{K,f_X}), \pi(y_{K,f_X})\rangle_{\ell_K}}{ \deg \phi}.
\end{align}
\end{proposition}
\begin{proof}
Let $\pi^\vee: J_X^\vee \to J_0(N)^\vee$ be the dual map on the dual abelian varieties and $\theta_J: J_0(N)^\vee \to J_0(N)$ be the principal polarization on $J_0(N)$.
Consider the polarization $\theta_X=\pi \circ \theta_J \circ \pi^\vee : J_X^\vee \to J_X$.  Let $\lambda: J_X^\vee \to J_X$ be the principal polarization of $J_X$. 
By identifying $\phi_* = \pi$ and $\phi^* =\theta_J \circ  \pi^\vee \circ \lambda^{-1}$ we see that $(\deg \phi) \lambda = \theta_X$.
In other words $\phi_* \phi^* = \id_{J_X} \deg \phi$, so 
letting \[e_X \colonequals \left( \frac{1}{\deg\phi} \right) \theta_J \circ \pi^\vee \circ \lambda^{-1} \circ \pi\]
we see that $e_X$ is an idempotent in $\End^0(J_0(N))$ which gives projection onto the component $\End^0(J_X)$. 

Proposition \ref{prop:ProjectionFormula} gives
 \begin{align*}
  \langle y_{K,f_X}, (\theta_J \circ \pi^\vee \circ \lambda^{-1})(\pi(y_{K,f_X})) \rangle_{\ell_K}  = \langle  \pi( y_{K,f_X}), \pi (y_{K,f_X}) \rangle_{\ell_K}.
  \end{align*}
 But the left hand side is also equal to $(\deg \phi) \langle y_{K,f_X},e_X( y_{K,f_X}) \rangle_{\ell_K} = (\deg \phi) \langle y_{K,f_X},y_{K,f_X} \rangle_{\ell_K} $, since $y_{K,f_X} = e_X( y_{K,f_X})$ and $e_X$ is idempotent.

 So $\langle y_{K,f_X},y_{K,f_X} \rangle_{\ell_K}=  \langle \pi(y_{K,f_X}),\pi(y_{K,f_X}) \rangle_{\ell_K} /(\deg \phi)$, as claimed.
\end{proof}

The following corollary follows from the formula for taking the trace \cite[(1.10.5)]{MazurTateBiext}. Let  $\ell_\Q $ be the cyclotomic character of $\Gal(\overline{\Q}/\Q)$.
\begin{corollary}
\label{cor:trace}
Let $\phi:X_0(N) \to X$ be a simple new $\Gamma_0(N)$-modular curve, with Jacobian $J_X$, and an associated newform $f_X$.
Let $\pi: J_0(N) \to J_X$ be induced by $\phi_*$.
Make an identification $\End^0(J_0(N)) \simeq E_{f_X}$.
We have
\begin{align*}
 \langle \pi(y_{K}),\pi(y_{K}) \rangle_{\ell_\Q} =  \frac{1}{2} \sum_{\sigma \in \Gal(E_{f_X}/\Q)} \langle \pi(y_{K,f^\sigma_X}) ,\pi(y_{K,f^\sigma_X})  \rangle_{\ell_K},
\end{align*}
where on the left hand side we are considering $\pi(y_{K})$ as a point in $J_X(\Q)$.
\end{corollary}
\begin{proof}
 We have that $y_K=  \sum_{g \in \New_N} \sum_{\sigma \in \Gal(E_g/\Q)}  y_{K, g^\sigma}$.  For $g \neq f_X$  the $g$-isotypical subspace of $J_0(N)$ is in the kernel of $\pi$, so
 \[ \langle \pi(y_K), \pi(y_K) \rangle_{\ell_K} = \sum_{\sigma \in \Gal(E_{f_X}/\Q)} \langle   \pi(y_{K, f^\sigma_X}),  \pi(y_{K, f^\sigma_X})\rangle_{\ell_K}.  \]
 Then by \cite[(1.10.5)]{MazurTateBiext}
 \[ [K:\Q] \langle \pi(y_K), \pi(y_K) \rangle_{\ell_K} = \langle 2 \pi(y_K),2 \pi(y_K) \rangle_{\ell_\Q}\] where $2 \pi(y_K)$ is the trace of $\pi(y_K)$ from $K/\Q$. 
\end{proof}

Perrin-Riou provides a comparison of the $p$-adic $L$-function described here to the $p$-adic $L$-function $L_{p, \MTT}(f)$ of Amice--V\'elu and Vishik \cite[(1.1)]{PerrinRiouInvent} discussed in the paper of Mazur, Tate, and Teitelbaum \cite{MazurTateTeitelbaum}. This comparison will be important computationally, since the latter $L$-function can be computed in Sage. We first give a brief description of the interpolation property of the $p$-adic $L$-function $L_{p, \MTT}(f)$; more details can be found in \cite{MazurTateTeitelbaum}.
Let $\Q_\infty$ be the cyclotomic $\Z_p$-extension of $\Q$ and $\Gamma_\Q \colonequals \Gal(\Q_\infty/\Q)$.
Let $\gamma$ be a topological generator for $\Gamma_{\Q}$.
Let $\zeta$ be a primitive $p^r$th root of unity. We write $\psi_\zeta$ for the associated character of $\Gamma_{\Q}$ sending $\gamma \mapsto \zeta$. We can also think of this as a Dirichlet character by considering it as a character of $\Gal(\Q(\zeta_{p^r})/\Q) \simeq (\Z/p^{r}\Z)^\times$. Write $\tau(\psi_\zeta)$ for the Gauss sum.
Then there exists an element
\begin{align}
\label{eqn:interpolationMTT}
\psi_\zeta(L_{p, \MTT}(f)) = e_p(\zeta) \frac{L(f, \psi_\zeta^{-1}, 1)}{\Omega_f^{\sgn(\psi_\zeta)}}
\end{align}
where $\Omega_f^{\pm}$ are certain periods associated to $f$, which we will elaborate on later (see \eqref{eqn:interpolationMTT} and Algorithm~\ref{alg:plusperiod}), and
\begin{align}
e_p(\zeta) = \left\{ \begin{array}{cc} \alpha_p^{- r - 1} \frac{p^{r+1}}{\tau(\psi_\zeta^{-1})} & \text{ if } \zeta \neq 1 \\ \alpha_p(f)^{-1}\left(1- \frac{1}{\alpha_p(f)}  \right)^2& \text{ if } \zeta = 1. \end{array} \right.
\end{align}

 Let $\Omega_f \colonequals 8 \pi^2 \|f \|$ be the period of the modular form $f$ \cite[p.458]{PerrinRiouInvent}. Let $\varepsilon$ denote the quadratic character associated to $K$ with conductor $|D|$. Then
\begin{align}
\label{PRcomp}
\ell_K(\calL_p(f, \mathbf{1})) &= \ell_\Q(L_{p, \MTT}(f))\ell_\Q(L_{p,\MTT}(f^{\varepsilon})) \left( \frac{ \sqrt{|D|}}{ \Omega_f} \right).
\end{align}

Sage has an implementation of the $p$-adic $L$-function of Amice--V\'elu and Vishik, so in practice, we use \eqref{PRcomp} and the fact that $L(f, 1) = 0$ to translate Theorem \ref{thm:PRpadicGZ} from a statement about the derivative of $\calL_p(f, 1)$ in the direction of $\ell$ into a statement about this $p$-adic $L$-function to compute the cyclotomic  $p$-adic height  of $y_{K,f}$:
\begin{align}
\label{eqn:derivPRMTT}
\calL'_{p, \ell}(f, 1) &= L'_{p, \MTT}(f,1) L_{p,\MTT}(f^{\varepsilon},1) \left( \frac{ \Omega^+_f   \Omega^+_{f\varepsilon}\sqrt{|D|}}{ \Omega_f} \right).
\end{align}

To compute $L_{p,\MTT}(f^{\varepsilon},1)$ we use the interpolation property of the $p$-adic $L$-function \eqref{eqn:interpolationMTT}:
 \begin{align}
 \label{eqn:hstexpanded}
L_{p,\MTT}(f^\varepsilon,1) =   (1 - 1/\alpha_p(f^\varepsilon))^{2} L(f^\varepsilon, 1)/\Omega_{f^\varepsilon}^+.
\end{align}
Since we chose $K$ to be a field where $p$ splits, $a_p(f^\varepsilon) = \varepsilon(p) a_p(f) = a_p(f)$ and therefore $\alpha_p(f^\varepsilon) = \alpha_p(f)$. We use the equality $(1 - 1/\alpha_p(f^\varepsilon))^{2} = (1 - 1/\alpha_p(f))^{2}$ to cancel some factors.

When we combine \eqref{eqn:hstexpanded} with \eqref{eqn:derivPRMTT}, we have
\begin{align}
\label{eqn:heightseq}
&\langle y_{K,f} ,y_{K,f} \rangle_{\ell_K} = \\  
 &\left( \frac{ \Omega^+_f   \Omega^+_{f\varepsilon} \sqrt{|D|}}{ \Omega_f} \right)\left( 1 - \frac{1}{\alpha_p(f)} \right)^{-2}
 \frac{L(f^\varepsilon, 1)}{\Omega_{f\varepsilon}^+}
\frac{d}{dT}  L_{p, \MTT}(f,T) \bigg|_{T = 0}\log_p(1+p). \notag
\end{align}
The conversion from $L_{p, \MTT}(f, s)$ to the series expansion $L_{p, \MTT}(f,T)$ requires a choice of topological generator $1+p$ for the Galois group of the cyclotomic $\Z_p$-extension $\Gal(K_\infty^{\cyc} / K)$.

Let $\sigma \in \Gal(E_f/\Q)$. By substituting $f^\sigma$ into the right hand side of \eqref{eqn:heightseq}, we obtain $\langle y_{K,f^\sigma}  ,y_{K,f^\sigma} \rangle_{\ell_K}$.

\begin{myalgorithm}[The cyclotomic $p$-adic height over $K$ of the $f$-isotypical component of the Heegner point $y_{K,f}$]
\label{alg:computeheightPR}
\hfill

\noindent Input: 
\begin{itemize}[itemsep=0pt]
 \item $f \in S_2(N)$ newform with coefficients in $E_f$; 
\item $K$ imaginary quadratic field of class number $1$ and discriminant $D<-3$ satisfying the Heegner hypothesis for $N$;
\item $p$ a prime split in $K$; and
\item an embedding $e: E_f \to \Q_p$ such that $f$ is ordinary in $e$.
\end{itemize}
Output: The cyclotomic $p$-adic height $\langle y_{K, f}, y_{K,f} \rangle_{\ell_K}$ over $K$ of $y_{K,f} \in J_0(N)(K)\otimes E_f$.
\begin{enumerate}[itemsep=0pt]
\item Compute $\frac{d}{dT} L_{p, \MTT}(f,T)\big|_{T = 0}$ using overconvergent modular symbols \cite{PSOverconvergent}.
\item Compute $ L(f^\varepsilon, 1)$ using Dokchitser's algorithms \cite{DokchitserLvals}.
\item Compute $\Omega^+_f$ using Algorithm \ref{alg:plusperiod} (normalized to agree with the normalization on the overconvergent modular symbols).
\item Compute $\|f\|$, for example, using \cite{DanCollinsPetersson}.
\item Return $\frac{\Omega_f^+ \sqrt{|D|} }{ 8 \pi^2 \|f\|}\cdot \left( 1 - \frac{1}{\alpha_p(f)} \right)^{-2} \cdot L(f^\varepsilon, 1) \cdot  \frac{d}{dT} L_{p, \MTT}(f,T) \big|_{T = 0} \log(1+p)$.
\end{enumerate}
\end{myalgorithm}

\begin{remark}
The convention of the sign of the height in Perrin-Riou differs from the convention chosen in Mazur--Tate--Teitelbaum and Pollack--Stevens. To achieve the correct normalization for $p$-adic BSD we negate the sign of the height returned by Algorithm~\ref{alg:computeheightPR}.
\end{remark}

To compute the quantity $\Omega_f^+$ we exploit the relationship \cite[I \S 8 (8.6)]{MazurTateTeitelbaum} between $f$ and quadratic twists of $f$ by fundamental discriminants $D'>0$. 
Let $\tau(\chi)$ denote the Gauss sum 
\begin{align}
\tau(\chi) \colonequals \sum_{a \mod D'} \chi(a) e^{2 \pi i a/D'}.
\end{align}
Since $D'$ is a fundamental discriminant,  $\tau(\chi) = \sqrt{D'}$.

Before we state the formula we need, we establish some background on modular symbols, following \cite{PSOverconvergent,PollackAWS}. Write $\Delta_0 \colonequals \Div^0(\P^1(\Q))$. We can act on $\Delta_0$ by elements of $\Gamma_0(N)$ via fractional linear transformation. 
Via this action, the set of additive homomorphisms $\Hom(\Delta_0, \C)$ has an action  \[ \varphi|\gamma \colonequals \varphi(\gamma E)\] where $\varphi \in \Hom(\Delta_0, \C)$, $E \in \Delta_0$, and $\gamma \in \Gamma_0(N)$. The \defi{$\C$-valued modular symbols} are those symbols that are invariant under the action of all $\gamma \in \Gamma_0(N)$, i.e. $\varphi|\gamma = \varphi$. We denote the space of these symbols as $\Symb_{\Gamma}(\C)$.

For any weight $2$ newform $g$ there exists a $\C$-valued modular symbol $\psi_g \in \Symb_{\Gamma}(\C)$ given by 
\[ \{ s\} - \{r\}  \mapsto  2 \pi i \int_r^s g(z) dz. \] In this context, $\{ s\} - \{r\}$ denotes the divisor with support $+1$ on $s \in \Q$ and $-1$ on $r \in \Q$.  This symbol encodes information about the twisted $L$-values of $g$. 

There is a $2$-dimensional subspace of $\Symb_{\Gamma}(\C)$ where the action of Hecke is equal to the eigenvalues of $g$.
The set $\Hom(\Delta_0, \C)$ has an involution $\iota =( \begin{smallmatrix}  -1 & 0 \\ 0 & 1  \end{smallmatrix})$ and so  $\psi_g$ can be decomposed as a sum of modular symbols $\psi_g = \psi_g^+ + \psi_g^-$. 
There exist complex numbers $\Omega_g^+$ and $\Omega_g^-$ such that $\varphi_g^+ \colonequals \psi_g^+/\Omega_g^+$ and $\varphi_g^- \colonequals \psi_g^-/\Omega_g^-$ take values in $E_g$ (see \cite[Theorem~2.2]{pAdicBSD}). 

The period $\Omega_g^+$ is only well-defined up to an element of $\overline{\Q}$ that is a unit in $\Q_p$. 
By \cite[I \S8 (8.6)]{MazurTateTeitelbaum} we can write the following relationship between modular forms and modular symbols
\begin{align}
\label{MTT:mainequality}
\frac{L(f^\chi, 1)}{\Omega_f^+} = \frac{\tau(\chi)}{D'} \sum_{\substack{a=1\\ \gcd(a,D')=1 }}^{ \lfloor D'/2 \rfloor} \chi(a)  
(  \varphi_f^+(\{a\} - \{D'\} ) + \varphi_f^+(\{ -a\} - \{ D'\})).
\end{align}
If $f^\chi$ is rank $0$, the sum will be non-zero. 

In practice, when evaluating the symbol $\psi_g^{\pm }$ and therefore the $p$-adic $L$-values of $g$ in  a computer algebra program, a choice of $\Omega_f^+$ must be fixed. For example, Sage makes a random choice of generator in the Hecke-eigenspace of $\Hom(\Delta_0, \C)^{\pm}$ corresponding to $g$ \cite[Section~3.5.3]{Stein}.
To extract this choice, we can evaluate the modular symbols $\varphi_f^+$ in \eqref{MTT:mainequality}.
This leads to the following algorithm for determining the period $\Omega_f^+$ that is compatible with the normalization on $\varphi_f^+$.

\begin{myalgorithm}[The plus period, normalized to agree with the overconvergent modular symbols]
\label{alg:plusperiod}
\hfill

\noindent Input: 
\begin{itemize}[itemsep=0pt]
\item $p$ a prime
\item $f \in S_2(N)$ newform with coefficients in $E_f$ and an embedding $e: E_f \to \Q_p$ such that $f$ is ordinary in $e$
\end{itemize}
Output: The period $\Omega_f^+$ (normalized to agree with the overconvergent modular symbols)

\begin{enumerate}[itemsep=0pt]
\item Set $D' \colonequals 5$.
\item \label{step:evalmodsym} Compute the right hand side of \eqref{MTT:mainequality} by evaluating the modular symbols and set $R$ equal to this value.
\item If $R$ is equal to $0$, set $D'$ to the next largest fundamental discriminant, and go back to Step \eqref{step:evalmodsym}.
\item Compute $L(f^\chi, 1)$ using Dokchitser's algorithms \cite{DokchitserLvals}.
\item  Return $R/L(f^\chi, 1)$.
\end{enumerate}
\end{myalgorithm}

\begin{remark}
When $f$ is the modular form associated to an elliptic curve $E$, we can take $\Omega_f^+$ to be the real period $\Omega_E^+$ of the elliptic curve.
\end{remark}

\begin{example}
\label{ex:height61}
Let $f_E$ be the modular form associated to the elliptic curve with LMFDB label \href{http://www.lmfdb.org/EllipticCurve/Q/61/a/1}{\texttt{61.a1}} and $p = 5$ a prime of good ordinary reduction. Let $\pi: X_0(61) \to E$ denote the modular parametrization. Choose $D = -19$ a Heegner discriminant for $E$. Then the $p$-adic $L$-series expansion for $E/\Q$ can be computed in Sage using \cite{PSOverconvergent}
\begin{align*}
 &L_{p, \MTT}(f_E,T) =
 O(5^{10}) + (1 + 2\cdot 5^2 + 5^3 + 5^4 + 3\cdot 5^5 + 2\cdot 5^7 + O(5^8))\cdot T  \\
 &+ (1 + 4\cdot 5 + 3\cdot 5^2 + 2\cdot 5^3 + 2\cdot 5^4 + 5^5 + O(5^6))\cdot T^2 + O(T^3).\notag
\end{align*}
Using the interpolation property, we have $\ell_\Q(L_{p,\MTT}(f^{\varepsilon})) =\left( 1 - \frac{1}{\alpha_p} \right)^2 L(f^\varepsilon_E, 1)/\Omega^+_{f^\varepsilon_E} $ and we can evaluate
\begin{align*}
L(f^\varepsilon_E, 1)/\Omega^+_{f^\varepsilon_E} =  2.
\end{align*}  
Finally, by \cite[Proposition~1]{CremonaModularDegree}, we have $\Omega_{f_E} =  2 \deg \pi \cdot \Vol E$, so
\begin{align*}
 \left( \frac{ \Omega^+_{f_E}   \Omega^+_{f^\varepsilon_E} \sqrt{|D|} \deg \pi}{ \Omega_f} \right) = \left( \frac{ \Omega^+_{f_E}   \Omega^+_{f^\varepsilon_E} \sqrt{|D|} }{ 2 \Vol E} \right) = 1.
 \end{align*}
Altogether, we evaluate the following formula for the $p$-adic height of $\pi(y_{K, f_E})$:
\begin{align*}
& \frac{1}{2} \langle \pi_E(y_{K, f_E}) , \pi_E(y_{K, f_E}) \rangle_{\ell_K} = \langle \pi_E(y_{K, f_E} ),\pi_E( y_{K, f_E}) \rangle_{\ell_\Q}  \notag \\
 &=\left( \frac{ \Omega^+_{f_E}   \Omega^+_{f^\varepsilon_E} \sqrt{|D|}}{ 2 \Vol E} \right)\left( 1 - \frac{1}{\alpha_p} \right)^{-2} \frac{
L(f^\varepsilon_E, 1)}{\Omega^+_{f^\varepsilon_E}}
\frac{d}{dT} L_{p, \MTT}(f_E,T)\bigg|_{T = 0}\log_p(1+p)\notag \\
&  = 4\cdot 5 + 4\cdot5^2 + 2\cdot 5^3 + 5^4 + 4\cdot 5^5 + 5^6 + 2\cdot 5^7 + 4\cdot 5^8 + O(5^9).\notag 
\end{align*}
Since $f_E$ has analytic rank $1$, by Gross--Zagier--Kolyvagin, the rank of $E(\Q)$ is one and the trace of the $f_E$-isotypical component of the Heegner point found here generates $E(\Q)$ up to finite index.
\end{example}

\begin{example}
\label{ex:height73}
Let $p = 11$ and $D =-19$.
Let $f$ and $f^\sigma$ be the modular forms in the newform orbit \href{https://www.lmfdb.org/ModularForm/GL2/Q/holomorphic/73/2/a/b/}{\texttt{73.2.a.b}}
given by 
\begin{align*} 
f &= q + (-\nu - 1)q^2 + (\nu - 2)q^3 + 3\nu q^4 + (-\nu - 1)q^5 + q^6 - 3q^7 + O(q^{8})\\
f^\sigma &= q + (\nu - 2)q^2 + (-\nu - 1)q^3 + (-3\nu + 3)q^4 + (\nu - 2)q^5 + q^6 - 3q^7 +  O(q^{8}).
\end{align*}
This has coefficient field $E_f = \Q(\nu)$ where $\nu$ has minimal polynomial $\nu^2 - \nu - 1$.
Then $f$ and $f^\sigma$ are associated newforms for the simple new $\Gamma_0(N)$-modular curve $X =X_0(73)^+$.
Fix the embedding $e: E_f\to \Q_p $
\begin{align}
\label{emb:pemb1}
\nu &\mapsto 8 + 7\cdot11 + 10\cdot11^2 + 7\cdot11^3 + O(11^4). 
\end{align}

By computing the derivatives of the $p$-adic $L$-functions in Sage, we get the values
\begin{align}
&\frac{d}{dT} L_{p, \MTT}(f,T) \bigg|_{T = 0} =   
7 + 7\cdot 11 + 2 \cdot 11^2 + 9 \cdot 11^3 + 4\cdot 11^4 +  O(11^5) \label{eqn:derivf}\\
&\frac{d}{dT} L_{p, \MTT}(f^{\sigma},T) \bigg|_{T = 0} =2 + 5 \cdot 11 + 6 \cdot 11^2 + 3 \cdot 11^3 + 8\cdot 11^4 + O(11^5)  \label{eqn:derivfsig}
.\end{align}
We have that $a_p(f) = \nu - 2$ and $a_p(f^\sigma) = -\nu - 1$, and we embed via $e$. It remains to compute the twisted $L$-value and the periods. We fix a complex embedding $e_c: E_f \to \C$ given by
\begin{align}
\label{emb:complex73}
\nu &\mapsto 1.618.
\end{align}
We can compute the Petersson norm of $f$ and $f^\sigma$ under \eqref{emb:complex73}. We get
\begin{align*}
 \| f \| = 0.986763 \text{ and } \| f^\sigma \| = 0.368434.
 \end{align*}
Changing the complex embedding would swap the values of the norms.

Using Dokchitser's package for computing values of $L$-functions we now compute $L(f^\varepsilon, 1)$ where $\varepsilon$ is the quadratic character twisting by $D = -19$. This yields $L(f^\epsilon, 1) = 
4.771908$
under the embedding $e_c$. 

Finally, we compute $\Omega_f^+$.
For $D' = 5$ we find that the right hand side of \eqref{MTT:mainequality} is $-4/\sqrt{5}$.
We compute $L(f^\chi, 1)  = 6.34683.$
Therefore $\Omega_f^+ = 3.5479.$
Combining the complex terms, we find
\begin{align*}
&\frac{ \Omega^+_f   L(f^\varepsilon, 1) \sqrt{|D|}}{ \Omega_f} = 0.94721.
\end{align*}
Numerically, by computing the quantities to higher precision, we recognize this as being close to an algebraic number having minimal polynomial $20x^2 - 20x + 1$, and so belongs to $E_f$. It appears to be $e_c(r_1)$ where $r_1 \colonequals 2/5 \nu + 3/10$.
Repeating the calculations for $f^\sigma$, we obtain $\Omega^+_{f^\sigma}   L(f^{\sigma \varepsilon}, 1) \sqrt{|D|}/ \Omega_f^{\sigma} =  r_2 \colonequals -2/5 \nu + 7/10$, the other root of this minimal polynomial.

The values $e(r_1)$, \eqref{eqn:derivf}, and $e(\alpha_p(f))$ can be combined using \eqref{eqn:heightseq} to obtain the height of $y_{K,f} \in J_0(N)(K)$. To project $\pi: J_0(73) \to J_0(73)^+$, we multiply by the degree of the quotient map $X_0(73) \to X_0(73)^+$  which is $2$. We have
\begin{align}
\label{ht1}
&\langle \pi ( y_{K, f}),\pi( y_{K,f} )\rangle_{\ell_K} = 6\cdot 11 + 5\cdot 11^2 + 3\cdot 11^3 + 8\cdot 11^5 + 7\cdot 11^6 + O(11^7).
\end{align}

Similarly, we can compute $\langle\pi( y_{K, f^\sigma}) ,\pi( y_{K,f^\sigma}) \rangle_{\ell_K} $ by using the quantities $e(r_2)$,  \eqref{eqn:derivfsig}, and $\alpha_p(f^\sigma)$ in the formula \eqref{eqn:heightseq}. We get
\begin{align}
\label{ht2}
&\langle\pi ( y_{K, f^\sigma}) ,\pi( y_{K,f^\sigma} )\rangle_{\ell_K} =
7\cdot 11 + 8\cdot 11^2 + 10\cdot 11^3 + 11^4 + 2\cdot 11^5 + 3\cdot 11^6 +O(11^7).
\end{align}
 Using Corollary \ref{cor:trace} we also obtain $\langle \pi(y_{K,f}),\pi(y_{K,f}) \rangle_{\ell_\Q}$.
\end{example}

\section{Quadratic Chabauty}
\label{ch:quadchab}
Let $X$ be a smooth projective geometrically integral curve defined over $\Q$ of genus $g>1$. Quadratic Chabauty \cite{QCI,QCII} is a technique for studying rational points on $X$ that computes a finite set of $p$-adic points containing $X(\Q)$ in some cases when the rank of $J$ is greater than or equal to the genus of $X$.

Quadratic Chabauty uses local and global $p$-adic height functions to construct a quadratic Chabauty function $\rho(z)$ that is used to cut out the finite set of $p$-adic points of $X$ containing $X(\Q)$. We will denote the global $p$-adic cyclotomic height on $y \in J(\Q)$ by $ h(y) \colonequals \langle y, y \rangle_{\ell_\Q}$ when the field of definition and choice of id\`ele class character is clear. Otherwise we will use the notation in Section \ref{Ch:Height}. 
The global height $h$ is a sum of local heights $h = \sum_\ell h_\ell$ where $\ell$ ranges over finite primes.
For $\ell \neq p$, the local height $h_\ell$ is a biadditive, continuous, and symmetric function on pairs of disjoint $\Q$-rational divisors of degree zero on $X$. For $\ell = p$, Coleman and Gross describe $h_p$ as a Coleman integral of a third kind differential form.
For more background on $p$-adic heights, which we do not describe in detail here, see \cite{ColemanGross,BBCrelle,QCIntegral}.

In Section~\ref{sec:qcintegral} we discuss the case of rank $1$ elliptic curve $E$ and construct a locally analytic quadratic Chabauty function $\rho(z)$ whose solutions contain the integer points of $E$.  In Section~\ref{sec:qcrational} we discuss the case of rational points on higher genus curves.  We give our main theorem, Theorem \ref{thm:mainthm}, explicitly constructing the quadratic Chabauty function $\rho(z)$ as a locally analytic function without knowing any infinite order points in the case of simple new $\Gamma_0(N)$-modular curves. 
We also provide several examples of how to apply Theorem \ref{thm:mainthm} in practice.

\subsection{Integral points on rank one elliptic curves}
\label{sec:qcintegral}
In this section, we study the case of determining integral points on a rank $1$ genus $1$ (elliptic) curve $E/\Q$. A consequence of Faltings's theorem is that the affine curve obtained by removing a point $\calX \colonequals E - \calO$ has finitely many integral points $\calX(\Z)$. Quadratic Chabauty for integral points on rank $1$ elliptic curves requires an infinite order point in $E(\Q)$. We replace this requirement with the computation of special values of two $p$-adic $L$-functions constructed by Perrin-Riou (see Section~\ref{Ch:Height}) and Bertolini, Darmon, and Prasanna (see Section~\ref{Ch:Log}) that determine the height and logarithm of a Heegner point for $E$, respectively. This allows us to determine $\calX(\Z)$ without knowing a rational point of infinite order.
Since we rely on the results from the previous sections, we make the assumptions \ref{hyp:logs} and \ref{hyp:heights}.

Let $E$ be a rank one elliptic curve over $\Q$ with conductor $N$ given by a Weierstrass equation
\[y^2 + a_1xy + a_3y = x^3 + a_2 x^2 + a_4 x + a_6.\]
Let $p> 2$ be a prime of good ordinary reduction. Let $\calE/\Z$ denote the minimal regular model of $E$ and $\calX = \calE - \calO$ the complement of the zero section in $\calE$.
Fix also differentials $\omega_0 = \frac{dx}{2y + a_1x +a_3}$ and $\omega_1 = x \omega_0$.

Let $b$ be a tangential basepoint at the point at infinity or an integral $2$-torsion point (see \cite[Section~1.5.4]{BesserHeidelberg} for more on tangential basepoints). Consider the two functions, the double Coleman integral 
\begin{equation}
D_2(z) \colonequals \int_b^z \omega_0 \omega_1
\end{equation}
as well as the logarithm, which can be expressed as the Coleman integral 
\begin{equation}
 \log(z)  \colonequals  \int_b^{z} \omega_0.
 \end{equation}

In \cite{BBCrelle,QCIntegral,EllipticQuadChabauty} Balakrishnan, Besser, and M\"uller  give algorithms to compute a finite set of $p$-adic points containing the integral points. We recall a related theorem.

Suppose that $p>3$ and let $E_2$ be the Katz $p$-adic weight $2$ Eisenstein series \cite{MazurSteinTate}. Define the constant 
\begin{equation}
\label{Katzconst}
c \colonequals \frac{a_1^2 + 4 a_2 - E_2(E, \omega_0)}{12}.
\end{equation} 
For any non-torsion point $P \in E(\Q)$ define $\gamma$ by
\begin{equation}
\gamma \colonequals \frac{h(P)}{\log(P)^2}.
\end{equation}

\begin{remark}
The quantity $\gamma$ does not depend on $P$. The $\Z_p$-module $\Z_p \otimes E(\Q)$ is $1$-dimensional, and has only a $1$-dimensional space of quadratic forms, and $\log(P) \neq 0$ when $P$ is non-torsion.
\end{remark}

\begin{theorem}[{\cite[Theorem~1.7]{BianchiBielliptic}}]
\label{Bianchithm}
Let $E$ be a rank $1$ elliptic curve over $\Q$ with good ordinary reduction at $p$ and bad reduction at the primes in a finite set $S$. There is a computable finite set $W \subset \Q_p$, $W = \prod_{q \in S} W_q$ such that $W_q$ is the possible local height contributions for an integral point at bad places, and $W_q$ is determined by the Kodaira type of the reduction of $E$ at $q$. For $w \in W$ define $\|w\|$ to be the sum of its elements.

If $E$ has good reduction at $q = 2$ or $q= 3$, and $\overline{E}(\F_q) = \{ \calO\}$, or if $E$ has split multiplicative reduction of Kodaira type $I_1$ at $2$, then $ \calX(\Z) = \emptyset .$

Otherwise, 
\begin{equation*}  
\calX(\Z) \subseteq  \bigcup_{w \in W} \psi(w),
\end{equation*}
where
\begin{equation}
\label{BianchiQCell}
\psi(w) \colonequals \{z \in \calX(\Z_p) : 2 D_2(z) + c \log(z)^2 + \| w \| = \gamma \log(z)^2\}. 
\end{equation}
\end{theorem}

We describe $\gamma$ in terms of two different special values of $p$-adic $L$-functions associated to $f \in S_2(N)$ the cusp form related to $E$ by modularity. 
This allows us to obtain new input into quadratic Chabauty as described in Theorem \ref{Bianchithm}, by replacing the $\gamma$ in \eqref{BianchiQCell} with one determined by special values of $L$-functions.
\begin{theorem}
\label{thm:mainthmelliptic}
Let $f$ be the modular form associated to $E$ and $\pi: X_0(N) \to E$ the modular parametrization. Assume \ref{hyp:heights} and \ref{hyp:logs}.
We have the equality
\begin{align} \label{eqn:gamma}
\gamma =   \frac{ \frac{1}{2} (\deg \pi) \left( 1 - \frac{1}{\alpha_p} \right)^{-4} \calL'_{p, \ell}(f, 1)}{ \left( \frac{1 - a_p(f) +p}{p}\right)^{-2} L_p(f,1)}
\end{align}
whenever  $L(f^\varepsilon, 1) \neq 0$.
Furthermore, $\gamma$ is computable. 

In other words, 
$\rho(z) = h_p(z) - \gamma\log(z)^2$ is a computable locally analytic function from $\calX(\Z_p)$ to $\Q_p$ that takes values on a finite computable set when evaluated on $\calX(\Z)$.
\end{theorem}
Theorem~\ref{thm:mainthmelliptic} allows us to determine a finite set of $p$-adic points of $\calX$ containing $\calX(\Z)$ without knowing an infinite order point of $E(\Q)$. 
\begin{proof}
Corollary \ref{cor:PRpadicGZ} shows that
\[\langle   \pi(y_{K,f}),\pi (y_{K,f}) \rangle_{\ell_K}  =  \deg \pi \calL_{p,\ell}'(f, 1)\left(1 - \frac{1}{(\alpha_p(f))}\right)^{-4}\] 
while \eqref{specialvalueBDP}  shows $L_p(f, 1)\left(\frac{1- a_p(f) +p}{p} \right)^{-2}$ is equal to $(\log_{f dq/q}\pi(y_K))^2$. 
Corollary~\ref{cor:trace} implies that 
$  \frac{1}{2} \langle   \pi(y_{K,f}),\pi (y_{K,f}) \rangle_{\ell_K}  = h( \pi(y_K)).$
\end{proof}

\begin{remark}
Suppose $X/\Q$ is a smooth projective geometrically integral curve of genus $2$ with Jacobian isogenous over $\Q$ to $E_1 \times E_2$ having Mordell--Weil rank $2$. These methods allow us to obtain a finite set of $p$-adic points containing $X(\Q)$ without knowing an infinite order point of the Jacobian. One simply follows the formula \cite[Theorem~1.4]{QCI} using the $\gamma_1$ and $\gamma_2$ in Theorem~\ref{thm:mainthmelliptic} associated with each $E_i$ as input.
\end{remark}

\begin{example}
Let $E$ be the elliptic curve with LMFDB label \href{http://www.lmfdb.org/EllipticCurve/Q/43/a/1}{\texttt{43.a1}} and consider $p = 11$ a prime of good ordinary reduction. This is a model for the modular curve $X_0(43)^+$. We choose $D = -7$ a Heegner discriminant for $E$ in which $p$ and $N = 43$ split. Fix a model for $E$
\[\calX:  y^2 + y = x^3 + x^2.\]
As in Examples \ref{ex:height61} and \ref{ex:log37}, we compute the constant $\gamma$:
\begin{align*}
\gamma 
&= \frac{h(\pi(y_{K,f}))}{\log(\pi(y_{K,f}))^2} =\frac{ \calL'_{p, \ell}(f, 1)   \left( \frac{1}{2} \right)\left( 1 - \frac{1}{\alpha_p} \right)^{-4} \deg \pi   }{ L_p(f, 1) \left( \frac{1 - a_p(f) +p}{p}\right)^{-2}}
\\
&= \frac{9\cdot 11 + 5\cdot11^2 + 5\cdot11^3 + 3\cdot11^4 + 7\cdot11^6 + 4\cdot11^7 + 4\cdot11^8 + O(11^9)}{11^2 + 8\cdot11^3 + 9\cdot11^4 + 6\cdot11^5 + 8\cdot11^6 + 6\cdot11^7 + 4\cdot11^8 + 4\cdot11^9 + O(11^{10})} \\
&= 9\cdot11^{-1} + 10 + 2\cdot11 + 4\cdot11^2 + 5\cdot11^4 + 8\cdot11^5 + 10\cdot11^6 + O(11^7).
\end{align*}
We proceed to solve the equations described by \eqref{BianchiQCell}.
The only prime of bad reduction for $E$ is $43$, and the Kodaira type of $E$ over $43$ is $I_1$ so $W = \{0\}$, and so $\calX(\Z) \subset \{ h_p(z) = 
\gamma\log(z)^2\}$.
Using a modified version of the code associated to \cite{BianchiBielliptic}, we obtain the finite set:
\begin{align*}
& \{(-1 , -1 ),
  (-1 , 0 ),
  (0 , -1 ),
  (0 , 0 ),
  (1 , -2),
  (1 , 1),
  (2 , -4),
  (2 , 3 ),
  (21, -99),
  (21 , 98), \\
&(10\cdot11 + 7\cdot11^2 + O(11^3) , 10 + 10\cdot11 + 9\cdot11^2 + 5\cdot11^3 + O(11^4) ),\\
  &(10\cdot11 + 7\cdot11^2 + O(11^3) , 11^2 + 5\cdot11^3 + O(11^4) ),\\
  &(1 + 6\cdot11 + 2\cdot11^2 + O(11^3) , 9 + 3\cdot11^2 + O(11^3) ),\\
 &(1 + 6\cdot11 + 2\cdot11^2 + O(11^3) , 1 + 10\cdot11 + 7\cdot11^2 + O(11^3) ),\\
  &(2 + 9\cdot11 + 7\cdot11^2 + O(11^3) , 3 + 8\cdot11 + 11^2 + O(11^3) ),\\
  &(2 + 9\cdot11 + 7\cdot11^2 + O(11^3) , 7 + 2\cdot11 + 9\cdot11^2 + O(11^3) )\}.
\end{align*}
This contains the $10$ integral points on $\calX$ as well as $3$ pairs of $\Z_{11}$-points conjugate under the hyperelliptic involution.
\end{example}

\subsection{Rational points on higher genus curves}
\label{sec:qcrational}

We now discuss how to extend the results of the previous section to the case of rational points on higher genus curves. 
For rank $1$ genus $1$ curves, we constructed the locally analytic function $\rho(z)$ used in quadratic Chabauty by writing the global height $h$ in terms of $\log_{\omega_0}(z)^2$, a locally analytic basis for $(H^0(X_{\Q_p}, \Omega^1)^\vee \otimes H^0(X_{\Q_p}, \Omega^1)^\vee)^\vee$. This strategy generalizes for finding rational points on higher genus curves.

However, the Coleman--Gross local height functions $h_v$ at each prime $v$ do not immediately extend to functions on $X(\Q_v)$; more sophisticated heights machinery is needed to deal with the heights of rational points. 
For this, we turn to Nekov\'a\v{r}'s theory of $p$-adic heights \cite{NekovarHeights,QCI} to define heights $h^{\Nek}$ and $h_v^{\Nek}$ of mixed extensions of Galois representations associated to points $x \in X(\Q_v)$.

For this section we assume the following.
\begin{assumption}
\label{hyp:qc} 
\hfill
\begin{enumerate}
\item  \label{hyp:simplenewmod} Assume $\phi: X_0(N) \to X$ is a simple new $\Gamma_0(N)$-modular curve;

\item assume $X/\Q$ has genus $g>1$;

\item assume $X(\Q) \neq \emptyset$ and fix a basepoint $b \in X(\Q)$;

\item 
\label{hyp:rankns} assume its Jacobian $J_X(\Q)$ has rank $r = g$;

\item \label{hyp:prime}
 assume $p$ is a prime of good reduction for $X$ such that $\log\colon J(\Q) \otimes \Q_p \to H^0(X_{\Q_p}, \Omega^1)^\vee$ is an isomorphism.

\end{enumerate}
\end{assumption}
Since we rely on the results from the previous sections, we also make the assumptions in \ref{hyp:logs} and \ref{hyp:heights}.
Note that by Assumption \ref{hyp:qc} \eqref{hyp:simplenewmod}, and Lemma \ref{thm:GL2typepartII}, we have $\rho(J_X) = r  = g $.   

Since $\rho(J_X) \geq 2$, we can find some nontrivial correspondence 
$Z \in \ker(\NS(J_X) \to \NS(X))$. Let $K = \Q$ or $\Q_p$. 
As explained in \cite[Section~5]{QCI}, the choice of $Z$ can be used to construct a certain quotient of the two step $\Q_p$-pro-unipotent fundamental group, and, by a twisting construction, for every $x \in X(K)$ we obtain an equivalence class of Galois representations 
\[A_Z(b,x) \in \{ G_{K} \to \GL_{2g+2}(\Q_p)\}/\sim.\] 
These Galois representations are \defi{mixed extensions}: 
they admit a $\G_K$-stable weight filtration with graded pieces 
$\Q_p(1), V \colonequals H^1_{et}(X_{\overline{K}}, \Q_p)^\vee, \Q_p$.  
Nekov\'a\v{r}'s theory of $p$-adic heights  \cite{NekovarHeights} yields a height function $X(\Q) \to \Q_p$ by sending $x \in X(\Q)$ to $h^{\Nek}(A_Z(b, x))$. 

Similar to the story for Coleman--Gross heights, the global height decomposes as a sum of local heights $h^{\Nek} = \sum_v h^{\Nek}_v$. Furthermore, $h^{\Nek}$ is bilinear in the following sense:
for each $z \in X(\Q)$ we have projection maps
\begin{align*}
 \pi_1( A_Z(b,z))  &= [W_0 A_Z(b,z) /W_{-2} A_Z(b,z)] \in H^1_f (G_{\Q}, V) \\
  \pi_2( A_Z(b,z))  &= [W_{-1}A_Z(b,z)] \in H^1_f (G_{\Q}, V^*(1))
\end{align*}
where $W$ denotes the weight filtration on the mixed extension. The height $h^{\Nek}$ is bilinear on $H^1_f (G_{\Q}, V) \times  H^1_f (G_{\Q}, V^*(1))$. Under our assumptions, these cohomology groups are both isomorphic to $H^0 (X_{\Q_p}, \Omega^1)^\vee$ and so $\pi_i$ can also be seen as a map into $H^0 (X_{\Q_p}, \Omega^1)^\vee$. 

The local height $h_p^{\Nek}$ at $p$ can be described in terms of linear algebraic data given by the filtered $\phi$-module associated to $A_Z(b,x)$. For more details see \cite[Section 3.3.2]{examplesandalg} or \cite[Section 4.3.2]{QCI}. We will simply write $h_p^{\Nek}(z)$ for the local height of $z \in X(\Q_p)$, omitting the dependence on $Z$ and $b$.

The quadratic Chabauty function $\rho(z)$ is equal to the difference between the local height at $p$ and the global height $h^{\Nek}(A_Z(b,z) ) - h_p^{\Nek}(z)$. 
The following theorem about $\rho(z)$ is the analog of Theorem \ref{Bianchithm}.  For this theorem it is not necessary to assume $X$ is a simple new modular $\Gamma_0(N)$-curve. 
\begin{theorem}[{\cite[Proposition~5.5]{QCI}}]
\label{mainQCtheorem}
Let $\psi_1, \dots, \psi_M$ be a basis for \\$(H^0(X_{\Q_p}, \Omega^1)^\vee \otimes H^0(X_{\Q_p}, \Omega^1)^\vee)^\vee. $
There are finite computable constants $\alpha_1, \dots, \alpha_M \in \Q_p$ such that the function $X(\Q_p) \to \Q_p$ given by
\begin{align*}
\rho(z) =  \sum_{i=1}^M \alpha_i \psi_i \circ (\pi_1, \pi_2 )(A_Z(b,z)) - h_p^{\Nek}(z) 
\end{align*}
is a locally analytic function. Furthermore, there exists a finite set $S \subset \Q_p$ such that $\{\rho(x) =s: x \in X(\Q_p), s \in S\}$ contains $X(\Q)$.
\end{theorem}
The set $S$ is given by computing local heights away from $p$ at primes $v$ of bad reduction. 
 When $X$ has a semistable regular model with geometrically irreducible special fibers, then $S = \{ 0\}$ \cite[Theorem~3.2]{examplesandalg}.

To solve for the $\alpha_i$ in Theorem \ref{mainQCtheorem}, our goal is to find constants such that  $\sum_{i=1}^M \alpha_i \psi_i \circ (\pi_1, \pi_2 )(A_Z(b,z)) = h^{\Nek}(A_Z(b,z))$, thus rewriting the global height as a locally analytic function.

Computing the $\alpha_i$ requires knowing \defi{sufficiently many rational points on $X$} \cite[Section~3.3]{examplesandalg}.  We need enough $z \in X(\Q)$ to find a basis of  $H^0 (X_{\Q_p}, \Omega^1) \otimes  H^0 (X_{\Q_p}, \Omega^1)$ of the form $(\pi_1(A_Z(b,z)), \pi_2(A_Z(b,z)))$ where $Z \in \ker(\NS(J_X) \to \NS(X))$ is nontrivial. The number of required rational points can also be decreased by working with symmetric $\End(J_X)$-equivariant heights.

If we do not have sufficiently many rational points, we can also use $r= g$ independent points on $J_X(\Q)$. Note that $(\pi_1(A_Z(b,z)), \pi_2(A_Z(b,z)))$ can be expressed in terms of a dual basis of $H^0 (X_{\Q_p}, \Omega^1)$, and $H^0 (X_{\Q_p}, \Omega^1)^\vee \simeq J_X(\Q)\otimes \Q_p$. The formula to represent $\pi_i(A_Z(b,z))$ in terms of the dual basis is given in \cite[(41)]{RecentApproaches}.
Furthermore, Besser \cite{BesserCGNekovar} gives an equivalence between the construction of the height due to Coleman and Gross and that of Nekov\'a\v{r}. In particular, they can be related through the study of a certain divisor, also studied in \cite{DRS}.
\begin{definition}
\label{defn:Dzb}
Define $D_{Z}(b, z)$ to be the degree zero divisor on $X$ given by $D_{Z}(b,z) \colonequals {Z}|_{\Delta} - {Z}|_{X \times b} - {Z}|_{z \times X} $.
\end{definition}
\begin{theorem}[{\cite[Theorem~6.3]{QCI}}]
\label{thm:cgnekheight}
Let $z \neq b$ be an element of $X(\Q)$.  Then
$h^{\Nek}(A_Z(b,z)) = h(z-b, D_{Z}(b,z))$.
\end{theorem}
 Therefore it is sufficient to express the height pairing in terms a basis for symmetric bilinear pairings on $J_X(\Q) \otimes \Q_p$ for a basis of $J_X(\Q)$. Then, given a choice of $Z$ and $b$ this determines a locally analytic function $h^{\Nek}: X(\Q_p) \to \Q_p$. 


This strategy for determining the height from Jacobian points is basis of the strategy we use in the proof of our main theorem, Theorem \ref{thm:mainthm}. However, because of the modular nature of our arguments, we do not need to explicitly describe a basis for $J_X(\Q)$. We also have a simplified calculation when computing the height pairing in terms of a basis of symmetric bilinear pairings on $J_X(\Q)$ because of the choice of dual basis. 

We now present a construction of $\rho(z)$ as a locally analytic function for simple new $\Gamma_0(N)$-modular curves that does not require knowing rational points on $X$ or $J_X$, other than the basepoint $b$. The theorem relies on Assumptions \ref{hyp:heights}, \ref{hyp:logs}, and \ref{hyp:qc}.
 This is the main result of this section.

\begin{theorem}
\label{thm:mainthm}
Let $\phi: X_0(N) \to X$ be a simple new $\Gamma_0(N)$-modular curve with associated $f \in \New_N$ of analytic rank $1$ and Jacobian $J_X$. Let $Z \in \ker(\NS(J_X) \to \NS(X))$ be a nontrivial correspondence and $b \in X(\Q)$ a choice of basepoint, and recall the divisor $D_Z(b,z)$ from Definition \ref{defn:Dzb}.

Let $p$ be a good prime that is ordinary for all $f^\sigma$, for $\sigma \in \Gal(E_f/\Q)$. Assume $p$ splits in $E_f$ and let $e$ be a choice of embedding $e: E_f \to \Q_p$. Recall that $\varepsilon$ is the quadratic character associated with the imaginary quadratic field $K$.

Define constants
\[\alpha_\sigma \colonequals \frac{ \frac{1}{2} (\deg(\phi) \calL_{p,\ell}'(f^\sigma,1)\left(1 - \frac{1}{e(\alpha_p(f^\sigma))}\right)^{-4}}{L_p(f^\sigma, 1)e(\left(\frac{1- a_p(f^\sigma) +p}{p} \right)^{-2})}\] for $\sigma \in \Gal(E_f/\Q)$. Assume $L(f^\varepsilon, 1) \neq 0$.

The $\alpha_\sigma$ are computable and for $z \in X(\Q)$ we have
\[\sum_{\sigma \in \Gal(E_f/\Q)} \alpha_\sigma (\log_{f^\sigma dq/q}(z))^2 = h(z).\] Hence
\[ \rho(z) =   \sum_{\sigma \in \Gal(E_f/\Q)} \alpha_{\sigma} \log_{f^\sigma dq/q}(z-b)\log_{f^\sigma dq/q}(D_Z(b,z)) - h_p^{\Nek}(z) \] is a locally analytic function on $X(\Q_p)$ away from $b$.
Furthermore, there exists a finite set $S \subset \Q_p$ such that $\{\rho(z) = s : z \in X(\Q_p), s \in S\}$ contains $X(\Q)$.
\end{theorem}
We give algorithms to compute the numerator and denominator of the constants $\alpha_\sigma$ appearing in Theorem \ref{thm:mainthm} in Section \ref{Ch:Height} and \ref{Ch:Log} respectively.

\begin{proof}
Since 
\[\{ f^\sigma  dq/q \text{ for } \sigma \in \Gal(E_f/\Q) \}\]
 is a basis for $H^0(X_{\Q_p}, \Omega^1)$
the functions 
\[\frac{1}{2} (\log_{f^\sigma  dq/q}(D ) \log_{f^\tau dq/q}(E)  + \log_{f^\sigma  dq/q}(E ) \log_{f^\tau dq/q}(D)) \text{ for }\sigma, \tau \in \Gal(E_f/\Q)\] form a basis for the symmetric bilinear pairings on $J_X(\Q) \otimes \Q_p$ by Assumption \ref{hyp:qc} \eqref{hyp:prime}.

We will show that for $z \in X(\Q)$, we have the equality $\rho(z) = h^{\Nek}(A_Z(b,z)) - h_p^{\Nek}(z)$, and in particular
\begin{align}
\label{eqn:heightsetting}
\sum_{\sigma \in \Gal(E_f/\Q)} \alpha_\sigma (\log_{f^\sigma dq/q}(z))^2 = h(z).\end{align}

Let $\pi: J_0(N) \to J_X$ be induced by the pushforward of $\phi$.
Corollary \ref{cor:PRpadicGZ} shows that
\[\langle   \pi (y_{K,f^\sigma}),\pi (y_{K,f^\sigma}) \rangle_{\ell_K}  =  \deg (\phi) \calL_{p,\ell}'(f^\sigma, 1)\left(1 - \frac{1}{e(\alpha_p(f^\sigma))}\right)^{-4}\] 
while \eqref{specialvalueBDP}  shows $L_p(f^\sigma, 1)e(\left(\frac{1- a_p(f^\sigma) +p}{p} \right)^{-2})$ is equal to $(\log_{f^\sigma dq/q}\pi(y_K))^2$. 

Corollary \ref{cor:trace} implies that 
\[  \sum_{\sigma  \in \Gal(E_f/\Q)}\alpha_\sigma (\log_{f^\sigma dq/q}\pi(y_K))^2 = h( \pi(y_K))\]
where $h:J(\Q)\to\Q_p$ is the global $p$-adic cyclotomic height of Coleman and Gross.

Recall from Section \ref{sec:Heegnerpoints} that Hecke acts via an order $\calO_f$ in $K$ and $\calO_f y_K$ generates a finite index subgroup of $J_X(\Q)$.  Consider the action of $\calO_f$ as through the embedding $e: \calO_f \to \Q_p$ so that $\calO_f \pi(y_K) \subseteq J_X(\Q)\otimes \Q_p$.
The bilinearity of the height implies that for all $C_1, C_2 \in \Q_p$,  
\[\langle C_1 \pi(y_K), C_2 \pi(y_K) \rangle_{\ell_\Q} = C_1 C_2  \langle \pi(y_K), \pi(y_K) \rangle_{\ell_\Q}.\]

Every $D \in J_X(\Q)$ can be written as $C \pi(y_K)$ for some $C \in \Q_p$.
Therefore, since the logarithm is linear 
 \[\langle \pi(y_K), \pi(y_K) \rangle_{\ell_\Q} = \sum_{\sigma \in \Gal(E_f/\Q)} \alpha_{\sigma} \log_{f^\sigma dq/q}(\pi(y_K))\log_{f^\sigma dq/q}(\pi(y_K)) \] implies  that 
  \begin{align}
  \label{eqn:heightsinJac}
  \langle D, E \rangle_{\ell_\Q} = \sum_{\sigma \in \Gal(E_f/\Q)} \alpha_{\sigma} \log_{f^\sigma dq/q}(D)\log_{f^\sigma dq/q}(E)
   \end{align}
  for all $D, E \in J_X(\Q)$.

Then $\log_{f^\sigma dq/q}(z)$ has a power series expansion in each residue disk, and is locally analytic on $X(\Q)$.
We can then extend $h^{\Nek}$ to a locally analytic function on $x \in X(\Q_p)$ away from $b$. By Theorem \ref{thm:cgnekheight}, we have the equalities \[h^{\Nek}(A_Z(b, z)) = h(\pi_1(A_Z(b, z)), \pi_2(A_Z(b, z))) = h(z-b, D_Z(b,z)).\] Then $z-b$ and $D_Z(b,z)$ can be viewed as elements of $J_X(\Q) \otimes \Q_p$ and therefore $h(z-b, D_Z(b,z))$ can be evaluated using \eqref{eqn:heightsinJac}, so 
\[h^{\Nek}(A_Z(b,z))= \sum_{\sigma \in \Gal(E_f/\Q)} \alpha_{\sigma} \log_{f^\sigma dq/q}(z-b)\log_{f^\sigma dq/q}(D_Z(b,z)).\]

For each $v \neq p$, the local height  $h_v^{\Nek}(X(\Q_v)) \subseteq S_v \subset \Q_p$ has finite image \cite{KimTamagawa} and we define $S \colonequals \{ \sum_v s_v : s_v \in S_v\}$. Then 
\[\rho(z) = h^{\Nek}(A_Z(b, z)) - h^{\Nek}_p(z) = \sum_{v \neq p} h^{\Nek}_v(z) \] and therefore
$\{\rho(z) = s : z \in X(\Q_p), s \in S\}$ contains $X(\Q)$.

Finally, $h_p(z)$ is the solution to a $p$-adic differential equation and therefore also locally analytic \cite[Lemma~3.7]{QCCartan}. Thus $\rho(z)$ is a locally analytic function on $X(\Q_p)$ away from $b$.
\end{proof}

\begin{example}
\label{ex:X067}
We consider the case of $X_0(67)^+$, a genus $2$ rank $2$ hyperelliptic curve. 
The rational points for $X_0(67)^+$ were previously determined in \cite{RecentApproaches}, but we give a new approach here.
Let $f$ and $f^\sigma$ be the newforms in the orbit \href{https://www.lmfdb.org/ModularForm/GL2/Q/holomorphic/67/2/a/b/}{67.2.a.b}. Then $E_f = \Q(\nu)$ where $\nu$ has minimal polynomial $z^2 - z - 1$. Let $f$ be the newform with the $q$-expansion
\begin{align*}
f(q) = q + (-\nu - 1)q^2 + (\nu - 2)q^3 + 3\nu q^4 - 3q^5 + q^6 + O(q^7).
\end{align*}
Let $p = 11$ and $D = -7$.
We fix the embedding $e: E_f \to \Q_p$ sending
$\nu \mapsto  4 + 3\cdot 11 + 3\cdot 11^3 + O(11^4).$ 

Using the methods of Section \ref{Ch:Log} we find that
\begin{align*}
\log_{f dq/q}(\pi(y_K))^2 &= 3\cdot 11^2 + 9\cdot 11^3 + 10\cdot 11^4 + 4\cdot 11^5 + 8\cdot11^6 + O(11^7),\\
\log_{f^\sigma dq/q}(\pi(y_K))^2 &= 11^2 + 11^4 + 11^5 + 9\cdot 11^6 + 6\cdot11^7 + O(11^8).
\end{align*}

To specify $\rho(z)$ in terms of a basis of symmetric bilinear forms on $J(\Q) \otimes \Q_p$, we need to relate the basis back to the coordinates of $X$.
Let $g_1$ and $g_2$ be a basis of modular forms for $S_2(67)_{\new}^+$ given by
\begin{align*}
g_1(q) &\colonequals q - 3q^3 - 3q^4 - 3q^5 + q^6  +  4q^7+ 3q^8 + 5q^9 - O(q^{10})\\
g_2(q) &\colonequals q^2 - q^3 - 3q^4  + 3q^7 + 4q^8 + 3q^9 - O(q^{10}). \notag
\end{align*}
We can construct a model $X$ of $X_0(67)^+$ over $\Q$ where under the identification $H^0(X, \Omega^1) \simeq S_2(67)_{\new}^+$, the differential $dx/y$ is $g_1$ and $xdx/y$ is $g_2$ by letting $x = g_2/g_1$ and $y = q dx/g_1$ and solving for the linear dependence in the monomials $1, x, x^2, \dots, x^6, y^2$.
The resulting model is 
\begin{align}
\label{eqn:model67ks}
X: y^2 = h(x) = 9x^6 - 14x^5 + 9x^4 - 6x^3 + 6x^2 - 4x + 1.
\end{align}
Then since $f = g_1 - (\nu +1) g_2$, we have 
\begin{align}
\label{reln:omegaf}
f dq/q = dx/y  - (\nu +1) x dx/y.
\end{align}
Let $Z$ be the trace zero correspondence associated to $T_p$. We fix an arbitrary choice of basepoint $b = [ 1: 1:2 ]$.
Write $A \colonequals 1$ and $B \colonequals - \nu - 1$. 

Using the techniques in Section \ref{Ch:Height} we get that
\begin{align*}
&\langle  \pi(y_{K, f }), \pi(y_{K, f })\rangle_{\ell_K} = 
4\cdot 11 + 3\cdot11^2 + 9\cdot11^3 + 11^5 + 3\cdot11^6 + 2\cdot11^7 + O(11^8)\\
&\langle \pi(y_{K, f^\sigma }),\pi(y_{K, f^\sigma }) \rangle_{\ell_K} = 10\cdot11 + 9\cdot11^3 + 8\cdot11^4 + 10\cdot11^5 + 7\cdot11^6 + 9\cdot11^7 + O(11^8)
\end{align*}
Let \[\alpha_1 \colonequals \frac{  \langle  \pi(y_{K, f }), \pi(y_{K, f })\rangle_{\ell_K} }{(\log_{f dq/q} \pi(y_K))^2 } \text{ and } \alpha_2 \colonequals \frac{\langle \pi(y_{K, f^\sigma }),\pi(y_{K, f^\sigma }) \rangle_v }{(\log_{f^\sigma dq/q } \pi(y_K))^2 }.\]
Then using linearity of the logarithm and \eqref{reln:omegaf}, by setting $\alpha_{00}  = \alpha_1A^2  +\alpha_2 A^{\sigma 2} $, $\alpha_{01} =2  ( \alpha_1 A B + \alpha_2 A^\sigma B^\sigma ) $ and $\alpha_{11} = \alpha_1 B^2 + \alpha_2 B^{\sigma 2} $ we obtain the relation
\begin{align*}
&\langle D,E \rangle_{\ell_{\Q}}=  \alpha_{00}   \log_{dx/y}(D) \log_{dx/y} (E) + \alpha_{01} \frac{1}{2} ( \log_{dx/y}(D) \log_{xdx/y} (E) +\log_{\omega_1}(D) \log_{dx/y} (E)  ) + \\
&\alpha_{11}  \log_{xdx/y}(D) \log_{xdx/y} (E) .
\end{align*}
 We use the basis $dx/y, xdx/y$ for convenience: this is the default basis in the code \cite{QCMod}.

Let $\rho = h^{\Nek} - h_p^{\Nek}$. We can construct $\rho$ as a locally analytic function and
solve for $\rho = 0$ using the code \cite{QCMod}: the input to this construction is the coefficients $\alpha_{ij}$. From this, the code writes $h^{\Nek}$ as a locally analytic function by expressing $\pi_i(A_Z(b, x))$ in terms of a dual basis. It also writes $h_p^{\Nek}$ as a locally analytic function. This process recovers the points found in \cite[Table~1]{RecentApproaches}.

Galbraith \cite{Galbraith} computed the Heegner point for some Atkin--Lehner quotients of modular curves; his computations show $\pi(y_K) = [1:-1:1] - [0:1:1]$ on the model \eqref{eqn:model67ks}. Using this, and forthcoming work of Gajovi\'{c} for computing local Coleman--Gross heights on even degree hyperelliptic curves we verified the above logarithm and height calculations.
\end{example}

\begin{example}
Let $f$ and $f^\sigma$ be the newforms in the orbit \href{https://www.lmfdb.org/ModularForm/GL2/Q/holomorphic/85/2/a/b/}{85.2.a.b} defined over $E_f = \Q(\sqrt{2})$.  Let
\begin{align*}
f = q + (\sqrt{2} - 1)q^2 + (-\sqrt{2} - 2)q^3 + (-2\sqrt{2} + 1)q^4 - q^5 - \sqrt{2}q^6 + O(q^7).
\end{align*}
We finish Example \ref{ex:log85} by studying the height the Heegner point $X^*_0(85)$. The rational points of $X^*_0(85)$ were determined by \cite{BarsGonzalesXarles}.

Let $p = 7$ and  $D = - 19$. In this example, we let $Z$ be the trace zero correspondence associated with $T_p$, and $b = [ 2: 38:5 ]$ on the model below. Recall we have fixed a $p$-adic embedding $e: E_f \to \Q_p$ by $\sqrt{2} \mapsto 3 + 7 + 2\cdot 7^2 + 6\cdot 7^3 +O(7^4).$
Using the methods described in Section \ref{Ch:Height} and already exhibited in the previous examples, we find
\begin{align*}
&\langle \pi(y_{K,f}), \pi(y_{K,f})\rangle_{\ell_K} = 
3\cdot 7^{-3} + 7^{-2} + 2 + 3\cdot7 + 3\cdot7^2 + 7^3 + 7^4 + 5\cdot7^5 + 7^6 +O(7^7)\\
&\langle \pi(y_{K,f^\sigma}),\pi(y_{K,f^\sigma}) \rangle_{\ell_K}  = 3\cdot7 + 6\cdot7^2 + 4\cdot5\cdot7 + 4\cdot 7^2 + 3\cdot7^3 + 2\cdot7^4 + 6\cdot7^6 + 5\cdot7^7  +O(7^8).
\end{align*}
Note that in this case the degree of the quotient $X_0(85) \to X^*_0(85)$ is $4$.

Picking a basis of newforms $g_1=q^2 - q^3+ \cdots, \, g_2=q - 3q^3 + \cdots$ with rational coefficients for the weight $2$ and level $85$ space of newforms with both Atkin--Lehner signs equal to $+1$, we can construct a rational model for $X_0(85)^*$ with $g_1 dq/q = dx/y$ and $g_2 = xdx/y$.
The model is
\[y^2 = x^6 - 4x^5 + 12x^4 - 22x^3 + 32x^2 - 40x + 25.\]
Furthermore we have the relationship
$(\sqrt{2} - 1)dx/y + xdx/y = f dq/q$.
On this model,  $(\log_{fdq/q}([1:-1:0] - [1:1:0]))^2$ agrees with $(\log_{fdq/q} \pi(y_K))^2$.
Furthermore, since $[2:5:1] - [2:-5:1]$ is linearly equivalent to twice $[1:-1:0] - [1:1:0]$, we can compute the global height of this divisor using forthcoming work of Gajovi\'{c} for computing local Coleman--Gross heights on even degree hyperelliptic curves:
\begin{align*}
&h([1:-1:0] - [1:1:0]) = \frac{1}{2} h_p([1:-1:0] - [1:1:0], [2:5:1] - [2:-5:1]) \\
&=5\cdot7^{-3} + 1 + 4\cdot7 + 6\cdot7^3 + 7^4 + 6\cdot7^5 + 3\cdot7^6 + 3\cdot7^7 + 3\cdot7^8 + 5\cdot7^9 + O(7^{10}).
\end{align*}
In this case, there are local height contributions away from $p$, which we did not compute.
\end{example}

\begin{example}
Let $f$ and $f^\sigma$ be the newforms in the newform orbit \href{https://www.lmfdb.org/ModularForm/GL2/Q/holomorphic/107/2/a/a/}{\texttt{107.2.a.a}} defined over $E_f= \Q(\nu)$ where $\nu$ satisfies the polynomial $z^2 - z - 1$. Let 
\begin{align}
f = q - \nu q^2 + (\nu - 2)q^3 + (\nu - 1)q^4 + (\nu - 2)q^5 + (\nu - 1)q^6  + O(q^7).
\end{align}
We consider the curve $X_0(107)^+$. The rational points of this curve were determined  \cite[Example~5.3]{examplesandalg}. In this case, there were not sufficiently many points on the curve to determine the height pairing, but one can use a pair of independent infinite order points on the Jacobian to set the Coleman--Gross height. We offer an alternative strategy.

Let $p = 11$ and  let $K$ be the imaginary quadratic field with discriminant $D = - 7$. 
Let $e: E_f \to \Q_p$ sending $\nu \mapsto  4 + 3\cdot 11 + 3\cdot 11^3 + O(11^4).$

Again, by picking a basis of newforms $j_1 = q - 2q^3 - q^4 + \cdots, \, j_2 = q^2 - q^3+ \cdots$ with rational coefficients for the weight $2$ and level $107$ space of newforms with Atkin--Lehner sign $+1$, we can construct a rational model for $X_0(107)^+$ with $j_1 dq/q = dx/y$ and $j_2 = xdx/y$.
Our model is
\[y^2 = x^6  - 10x^5+ 17 x^4- 18x^3 + 10x^2  - 4x  +1.\]
On this model $dx/y - \nu xdx/y = f dq/q$. Let $Z$ be the trace zero correspondence associated to $T_p$, and $b= [ 1: 1:2]$.
We can find a finite index subgroup of the Mordell--Weil group generated by the classes of
$Q_1 \colonequals [0:1:1]- [0:-1:1]$ and $Q_2 =[1/2:-1/8:1] - [0:-1:1]$.
The logarithms $L_{Q_i}\colonequals \log_{f dq/q} Q_i$ (under the embedding $e$) are
\begin{align*}
L_{Q_1}&=3\cdot11 + 4\cdot11^2 + 2\cdot11^3 + 9\cdot11^4 + 11^5 + 10\cdot11^7 + 7\cdot11^8 + 4\cdot11^9 + O(11^{10})
\\
L_{Q_2}&= 2\cdot11 + 8\cdot11^2 + 7\cdot11^4 + 4\cdot11^5 + 6\cdot11^6 + 3\cdot11^7 + 3\cdot11^8 + O(11^{10}).
\end{align*}
We also can compute the logarithm of the Heegner point using the techniques described in Section \ref{Ch:Log}.
The values of $\ell(r)$ are given in Table \ref{107emb2}.

\begin{table}[h]
\begin{subtable}[h!]{0.45\textwidth}
\begin{center}
\begin{tabularx}{120pt}{cr}
$r$ & $\ell(r) \mod \pp^{5}$\\
\midrule
$10$  & $-22250  $\\
$20$  & $-17899 $\\
$30$ & $-70252$\\
$40$ &$28890$ \\
$50$ & $56376$ \\
\end{tabularx}
\caption{$\ell(r)$ for $f$} 
\label{107emb2}
\end{center}
\end{subtable}
\begin{subtable}[h!]{0.45\textwidth}
\begin{center}
\begin{tabularx}{120pt}{cr}
$r$ & $\ell(r) \mod \pp^{5}$\\
\midrule
$10$  & $39142  $\\
$20$  & $70280 $\\
$30$ & $39031$\\
$40$ &$-40900$ \\
$50$ & $49703$ \\
\end{tabularx}
\caption{$\ell(r)$ for $f^\sigma $}
\label{107emb1}
\end{center}
\end{subtable}
\caption{Computations for $f$ in  \href{https://www.lmfdb.org/ModularForm/GL2/Q/holomorphic/107/2/a/a/}{\texttt{107.2.a.a}}}
\end{table} 
Then $L_p(f, 1)=-50731 +~O(11^5)$ and the logarithm of the Heegner point is
\[(\log_{f dq/q}\pi(y_K))^2 = 4\cdot11^2 + 8\cdot11^4 + 2\cdot11^6 + O(11^7).\]
Then we can check that when $A = \pm 2$ and $B = \pm 1$ we have the relation
\[A^2 (L_{Q_1}/2 )^2 + 2AB (L_{Q_1}/2)( L_{Q_2}/2) + B^2( L_{Q_2}/2)^2 = (\log_{f dq/q} \pi(y_K))^2 .\] 
The division by two occurs because the $Q_i$ are not $2$-saturated in the Mordell--Weil group.
For the conjugate modular form $f^\sigma dq/q$ we have the values given in Table~\ref{107emb1} and so
\[L_p(f^\sigma, 1) = 37471 + O(11^5)\] and therefore 
\[(\log_{f^\sigma dq/q} \pi(y_K))^2 = 3\cdot11^2 + 4\cdot11^3 + 2\cdot11^5 + 10\cdot11^6 + O(11^7).\]

We finish the example by computing the heights of $\pi(y_{K,f})$.
To do this we fix the complex embedding $e_c: E_f \to \C $ by $\nu \mapsto 1.61803$
First, using \cite{DanCollinsPetersson} we can compute
\begin{align*}
 \Omega_f = 42.114698 \hspace{.5in}
 \Omega_{f^\sigma} = 51.071742.
 \end{align*}
In Sage, we can also compute
\begin{align*}
&\frac{d}{dT} L_{p, MTT}(f,T)\bigg|_{T = 0}  =
 4 + 7\cdot11 + 7\cdot11^2 + 7\cdot11^3 + 8\cdot11^4 + 11^5 + 3\cdot11^6 + O(11^7)\\
&\frac{d}{dT} L_{p, MTT}(f^\sigma,T) \bigg|_{T = 0}  =
6 + 11 + 6\cdot11^2 + 9\cdot11^3 + 7\cdot11^5 + 7\cdot11^6 +  O(11^7).
\end{align*}
Let $\chi$ be the quadratic character with $D' = 5$ and recall that $\varepsilon$ denotes the quadratic character associated with $K$.
Following Algorithm \ref{alg:plusperiod} we find that the periods are
 $\Omega_f^+ = L(f^\chi, 1)/(-4 / \sqrt{5})$ and  $\Omega_{f^{\sigma}}^+ = L(f^{\sigma \chi}, 1)/( (-4) / \sqrt{5})$. 
We compute 
\begin{align*}
&L(f^\chi, 1)= 3.948128 \hspace{.5in}
L(f^{\sigma \chi}, 1) =2.407825\\
&L(f^\varepsilon, 1) = 0.996703 \hspace{.5in}
L(f^{\sigma\varepsilon}, 1) = 5.188648.
\end{align*}
Combining the complex values, we find
\begin{align}
\label{eqn:complexval}
&\frac{ \Omega^+_f   L(f^\varepsilon, 1) \sqrt{|D|}}{ \Omega_f}   = -0.138196 \hspace{.5in}
&\frac{ \Omega^+_{f^\sigma}  L(f^{\sigma\varepsilon}, 1) \sqrt{|D|}}{ \Omega_{f^\sigma}}  =-0.361803.
\end{align}
Let $r_1 \colonequals 1/10 \nu - 3/10$ and $r_2 \colonequals -1/10 \nu - 1/5$ be the roots of the polynomial $20x^2 + 10x + 1$.
The complex values of \eqref{eqn:complexval} computed to $30$ digits are within $10^{-28}$ of the algebraic numbers $e_c(r_1)$ and  $e_c(r_2)$.
Under the assumption that the values of \eqref{eqn:complexval} are $e_c(r_1)$ and $ e_c(r_2)$, we have
\begin{align*}
&\langle \pi(y_{K, f } ), \pi(y_{K,f }) \rangle_{\ell_K} =
8\cdot 11 + 4\cdot 11^2 + 10\cdot 11^3 + 7\cdot 11^4 + 4\cdot 11^5 + 8\cdot 11^6 + O(11^7) .\\
&\langle\pi( y_{K, f^\sigma} ),\pi( y_{K,f^\sigma} ) \rangle_{\ell_K} =6\cdot 11 + 11^2 + 2\cdot11^3 + 2\cdot 11^5 + 2\cdot 11^6 + O(11^7).
\end{align*}
We ran the quadratic Chabauty code \cite{QCMod} using the resulting $\rho(z)$ from Theorem~\ref{thm:mainthm} and recovered a finite superset of the rational points on $X_0(107)^+$.
\end{example}

\section*{Acknowledgments}
It is a pleasure to thank my thesis committee Jennifer Balakrishnan, Steffen M\"uller, Robert Pollack, and David Rohrlich for many helpful conversations and comments. I am very grateful to Alex Best, Francesca Bianchi, Edgar Costa, Henri Darmon, Stevan Gajovi\'{c}, Borys Kadets, Ari Shnidman, and John Voight for insightful correspondence and help with computations. I am also indebted to the anonymous referee for many helpful comments.

\newcommand{\etalchar}[1]{$^{#1}$}

\end{document}